\documentclass[11pt]{article}

\usepackage[a4paper,margin=1in]{geometry}
\usepackage{graphicx}
\usepackage{tikz}
\usepackage{amsmath,amssymb,amsthm,mathtools}
\usepackage{mathrsfs}
\usepackage{enumitem}
\usepackage{verbatim}
\usepackage{booktabs}
\usepackage{placeins}
\usepackage{hyperref}

\hypersetup{
  colorlinks=true,
  linkcolor=blue,
  citecolor=blue,
  urlcolor=blue,
  pdftitle={Merge Trees of Length-Filtered Lattice Knot Spaces},
  pdfauthor={Makoto Ozawa}
}

\usepackage{aliascnt}
\newtheorem{theorem}{Theorem}[section]
\newaliascnt{proposition}{theorem}
\newtheorem{proposition}[proposition]{Proposition}
\aliascntresetthe{proposition}
\newaliascnt{lemma}{theorem}
\newtheorem{lemma}[lemma]{Lemma}
\aliascntresetthe{lemma}
\newaliascnt{corollary}{theorem}
\newtheorem{corollary}[corollary]{Corollary}
\aliascntresetthe{corollary}
\newaliascnt{conjecture}{theorem}
\newtheorem{conjecture}[conjecture]{Conjecture}
\aliascntresetthe{conjecture}
\newaliascnt{problem}{theorem}
\newtheorem{problem}[problem]{Problem}
\aliascntresetthe{problem}
\newaliascnt{observation}{theorem}
\newtheorem{observation}[observation]{Observation}
\aliascntresetthe{observation}

\theoremstyle{definition}
\newaliascnt{definition}{theorem}
\newtheorem{definition}[definition]{Definition}
\aliascntresetthe{definition}
\newaliascnt{example}{theorem}
\newtheorem{example}[example]{Example}
\aliascntresetthe{example}
\newaliascnt{remark}{theorem}
\newtheorem{remark}[remark]{Remark}
\aliascntresetthe{remark}

\usepackage[capitalize,nameinlink]{cleveref}
\crefname{problem}{problem}{problems}
\Crefname{problem}{Problem}{Problems}

\newcommand{\Rop}{\operatorname{Rop}}
\newcommand{\Thi}{\operatorname{Thi}}
\newcommand{\Len}{\operatorname{Len}}
\newcommand{\Isom}{\operatorname{Isom}}
\newcommand{\Isomp}{\operatorname{Isom}^+}
\newcommand{\Conn}{\operatorname{Conn}}

\newcommand{\SC}{\operatorname{SC}}
\newcommand{\BFACF}{\operatorname{BFACF}}
\newcommand{\seed}{\mathrm{seed}}

\title{Merge Trees of Length-Filtered Lattice Knot Spaces}

\author{Makoto Ozawa}
\date{}

\begin{document}

\maketitle

\begin{abstract}
We study length-filtered BFACF graphs of self-avoiding polygons in the
simple cubic lattice, emphasizing the birth and merging of connected
components and the minimax length required to connect prescribed states.
For the published shortest-layer censuses of $4_1$ and $6_3$, we
compute BFACF merge trees: the former has four components at length
$30$ merging at $32$, while the latter has twelve components at $40$,
two at $42$, and one at $44$.  Completeness for the entire minimal
layer is conditional on the published populations exhausting it
(and on $40$ being minimal for $6_3$).
For one trefoil chirality, we certify six selected nonminimal components
born at length $26$, with merge heights $28$ or $30$ relative to the
minimal branch; completeness of this selected-birth tree is not asserted.
We characterize eventual connectivity of capped BFACF graphs in terms
of bounded birth levels of strong local-minimum components.
An explicit $64$-edge Hamiltonian lattice unknot has no neutral or
shortening BFACF move, and a verified BFACF path gives exact merge height
$66$ with the square.  Two-edge splicing produces frozen unknots of
length $64m$ for every $m\ge1$.
We further splice a minimal $24$-edge trefoil to a frozen Hamiltonian block,
obtaining an $88$-edge frozen trefoil.  Repeated protected-port splicing and
reflection give frozen trefoils of both chiralities of length $64m+24$.
Finally, a block substitution replaces every vertex of an arbitrary lattice
polygon by a $5\times5\times5$ Hamiltonian block with face-centred ports.
For every knot type $K$ and every lattice polygon $Q$ of type $K$ with $m$
edges, this gives a BFACF-frozen polygon of type $K$ and length $125m$.
Consequently every knot type, of either chirality, has infinitely many
nonminimal strong local-minimum births, lattice Gordian representatives,
and capped BFACF graphs that are not eventually connected.  Explicit
coordinate data and independent BFACF local-move verifiers accompany the
constructions.  The exact merge levels of the frozen families, unbounded
minimax excess, and analogous barriers in continuous ropelength-filtered
knot spaces remain open.
\end{abstract}

\noindent\textbf{Keywords.}
Lattice knots, BFACF moves, merge trees, strong local minima, Gordian
unknots, Hamiltonian polygons, block substitution, connected sums, persistent homology, minimax length,
ropelength.

\medskip

\noindent\textbf{2020 Mathematics Subject Classification.}
57K10; 57K12; 05C10; 68R10; 49Q10.

\section{Introduction}

The ropelength-filtered approach to knot theory studies a knot type through the
geometry of its thick representatives.  For a knot type \(K\) and a parameter
\(\Lambda>0\), consider the space
\[
Y_\Lambda(K)
=
\{\gamma\in K\mid \Thi(\gamma)\geq1,\ \Len(\gamma)\leq \Lambda\}/\Isom^+(\mathbb R^3).
\]
in the normalisation of \cite[Definition~6.1]{OzawaRFKS}.
Two representatives are called admissibly equivalent at scale \(\Lambda\) if they
can be joined by a path remaining in \(Y_\Lambda(K)\).  Thus the admissible
components are the path components \(\pi_0(Y_\Lambda(K))\), and as \(\Lambda\)
increases these components form a persistence object.  In the continuous theory
the path components and the connected components \(\Conn(Y_\Lambda(K))\) must in
principle be distinguished, since the compactness argument controls the latter;
by \cite[Theorem~6.20]{OzawaRFKS} the two merge functions nevertheless agree.
In the finite move graphs of the present paper the distinction is vacuous,
connected and path components of a finite graph being the same thing.  The first level at which
the space is non-empty is the ropelength level \(\Rop(K)\), and the initial layer
\[
I(K)=Y_{\Rop(K)}(K)
\]
may be regarded as the ideal stratum of \(K\).  This continuous
ropelength-filtered framework, including ideal strata, admissible components,
and merge scales, was developed in \cite{OzawaRFKS}.  The present paper
is intended as an experimental, computable lattice analogue of that framework.
For the reader's convenience, the constructions needed below are recalled in
self-contained form; no result of \cite{OzawaRFKS} is used as a black box
except for the motivating terminology and comparison with the continuous theory.

This continuous theory is conceptually natural, but it is difficult to compute.
Even deciding the number of admissible components of \(Y_\Lambda(K)\) at a fixed
scale is generally out of reach; only finitely many components survive from any
strictly lower level \cite[Corollary~6.22]{OzawaRFKS}, but finiteness at a
single level is open.  Likewise, the merge scale of two components,
namely the first value of \(\Lambda\) at which they become joined in the filtered
space, is not directly computable from the definition.

The aim of this paper is to organize these ideas in a discrete model in which
they can be implemented by finite graph algorithms.  Move graphs themselves are
not new: Reidemeister graphs of knot diagrams have been studied systematically by
Barbensi and Celoria \cite{BarbensiCeloria}, and computational work on hard unknot
diagrams shows that local-move simplification may require a temporary increase in
crossing number \cite{BurtonHardUnknot}.  Likewise, in lattice-knot theory the
BFACF algorithm and minimal-length enumerations have a substantial history
\cite{BFACF,DiaoMinimal,DiaoSmallest,Scharein2009,MinimalKnotsRechnitzer}.
Our purpose is to combine a size filtration with these finite move graphs and to
measure exact minimax length barriers for prescribed lattice representatives.
The basic replacement is
\[
\text{admissible deformation}
\quad\rightsquigarrow\quad
\text{finite sequence of local moves},
\]
and
\[
\text{admissible component}
\quad\rightsquigarrow\quad
\text{connected component of a finite move graph}.
\]
The resulting framework is called \emph{discrete knot theory} in this paper.

The experimental point of view is essential here.  The simplest useful objects
are not only the full graphs \(G_N\), which may be too large to enumerate, but
also the subgraphs generated from specified seed representatives.  These
seed-generated graphs are reproducible finite approximations to local regions of
the filtered state space.  They allow one to ask concrete questions such as:
when do two mirror-related lattice representatives first become connected after
allowing temporary length expansion?

Our first informative test case is the figure-eight knot.  Earlier work already
enumerated minimal simple-cubic representatives and partitioned minimal layers
into classes under length-preserving (type-0) BFACF moves
\cite{Scharein2009}; consequently, disconnectedness inside a minimal layer is not
by itself a new phenomenon.  We ask the next minimax question globally over the
published minimal layer.  With reflections retained, the $152$ minimal $4_1$
classes form four type-$0$ components of sizes $58,58,18,18$ at $N=30$, and an
exhaustive level-$32$ search contains all $152$ minimal classes in one bounded
component.  Thus its reported-layer merge tree is $4\to1$ at excess
length two.  For the amphichiral knot $6_3$, the $148$ minimal classes form
twelve type-$0$ components at $N=40$; exhaustive searches at $N=42$ show that
these merge into exactly two mirror-related components, and a verified BFACF
certificate joins those components at $N=44$.  Hence the published-layer
merge tree is $12\to2\to1$ (conditional on census completeness as a
statement about the entire minimal layer).  The resulting barrier hierarchy is a global
statement for the published collections under the specified BFACF system and
proper-cubic quotient.  Published minimal-length and population data are
not all established by independent exhaustive mathematical proofs: the
minimum length $30$ for $4_1$ is known, but the population $3648$ is taken
from a numerical census; the asserted minimum length $40$ and population
$3552$ for $6_3$ are likewise sourced from the reported census.
Thus the merge heights are unconditional for the explicitly verified published
representatives, while a description of the entire minimal layer is
conditional on completeness of the underlying census.  The numerical values remain lattice- and
move-system-dependent, and no universality beyond this model is claimed.

A second computation follows births above the minimal layer.  For a fixed
trefoil chirality, the three minimal components merge at length $26$, but
six certified components are born at that same level.  Five join the
minimal branch at $28$; the remaining $20$-state component joins it only
at $30$.  Thus nonminimal trefoils can require an excess of four edges to
reach any minimal representative.  The lower bound is witnessed by two
disjoint, completely verified components at cap $28$, rather than by a
failed or truncated search.  These are statements about specified components
of the full filtration.  Their exhaustiveness at length $26$ is separated
from their individual certification; no enumeration of all higher births
is asserted.

The main example uses lattice knots.  Let \(\mathcal L\) be a regular cubic
lattice, for instance the simple cubic lattice \(\mathbb Z^3\).  A lattice knot
is a self-avoiding closed polygon whose edges are lattice edges.  For a knot type
\(K\), let
\[
\mathcal P_N^{\mathcal L}(K)
=
\{\omega\subset \mathcal L
\mid
\omega\text{ is a lattice polygon of type }K,\ \ell(\omega)\leq N\}/\Isomp(\mathcal L),
\]
where \(\ell(\omega)\) is the number of lattice edges and \(\Isomp(\mathcal L)\) is the orientation-preserving lattice isometry group.
Thus spatial reflections are not identified, even for amphichiral knot types.
One may instead study a separate full-isometry quotient if mirror images are
intentionally identified; none of our component counts uses that quotient.  Choose a finite set of local moves, for example
BFACF-type moves or cubulated moves.  We then define a
finite graph
\[
G_N^{\mathcal L}(K)
\]
whose vertices are the elements of \(\mathcal P_N^{\mathcal L}(K)\) and whose
edges are given by the allowed moves.  The connected components
\[
\pi_0(G_N^{\mathcal L}(K))
\]
are the lattice admissible components at level \(N\).

This construction should not be interpreted as a canonical discretization of
ropelength.  A lattice polygon has corners and does not literally satisfy
\(\Thi=1\) as a smooth curve.  Rather, the lattice-filtered graph is a computable
combinatorial surrogate for the ropelength-filtered space.  It captures the same
structural pattern: representatives are filtered by size, deformations are
constrained by the same size bound, and components merge when the size bound is
increased.

There is a second, equally important qualification.  The connected components
under a fixed cap are components for the chosen move system.  In the computations
below this move system is the standard BFACF system.  Globally, without a length
cap, BFACF is ergodic on each simple-cubic knot type \cite{BFACF}; hence every
pair of representatives of the same knot type eventually merges.  This global
ergodicity does \emph{not} imply connectivity under a prescribed cap \(N\), and
it does not make a seed-generated search exhaustive over the full bounded graph.
The quantity studied here is precisely the least cap at which connectivity first
occurs.  A different complete move system could produce a different minimax
barrier.

A basic question is whether the full filtration is \emph{eventually connected}:
does some $N_0(K)$ satisfy that $G_N^{\SC,\BFACF}(K)$ is connected for every
$N\geq N_0(K)$?  Pairwise BFACF ergodicity does not answer this question,
because new components may be born at higher lengths after all previously
selected branches have merged.  In \cref{prop:eventual-connectivity} we show
that eventual connectivity is equivalent to bounded birth levels, or
equivalently to finiteness of the strong local-minimum components.
The Hamiltonian-block construction of
\cref{subsec:frozen-unlinks} proves that eventual connectivity fails
for the unknot.  The protected-port connected-sum construction of
\cref{subsec:frozen-trefoil} proves the same for both trefoil
chiralities.  The block substitution of \cref{subsec:frozen-all}
proves it for every knot type (\cref{cor:every-knot}).  The growth of
minimax excess remains open.

The paper also separates the present discrete theory from a finite
recognition theory based on diagrams and Reidemeister moves.  Finite
diagrammatic recognition data are certificates extracted from paths, whereas
the lattice-filtered graphs studied here are finite state spaces whose
components can be computed directly.  The two are related through the
conditional realizability statements of
\cref{sec:certificates,sec:realization}, but they are not identified.
Likewise, lattice edge number and Euclidean ropelength are different
resources, and any comparison between them is made only through explicitly
controlled maps (\cref{sec:controlled-comparison}).

\noindent\textbf{Organization.}
\Cref{sec:abstract,sec:initial,sec:lattice} give the abstract filtered-graph
formalism, initial layers and merge scales, and the lattice specialization.
\Cref{sec:scope} fixes the scope of the discrete model.
\Cref{sec:certificates} records the conditional PL-realizability statement
needed to turn realized discrete paths into finite Reidemeister
certificates, and \cref{sec:realization} discusses the converse realization
problem.  \Cref{sec:ropelength} explains the coarse relation with
ropelength-filtered spaces and states the convergence problem, and
\cref{sec:computability} discusses computability.
\Cref{sec:prototype,sec:experiments} describe the prototype implementation
and the reproducible experiments, including the reported shortest-layer
merge trees of \(4_1\) and \(6_3\), the certified nonminimal trefoil
branches, the frozen lattice unknots and trefoils, and the block
substitution producing frozen polygons of every knot type.  \Cref{sec:controlled-comparison}
states the controlled-simulation principle under which discrete merge data
can be compared with a geometric deformation space, and
\cref{sec:further} records problems suggested by the computations.

\subsection{Relation to Gordian phenomena and simplification}
\label{subsec:gordian-literature}

The word \emph{Gordian} refers to several different obstructions, whose
constraints must be specified.  Piera\'nski, Przyby\l{} and Stasiak
\cite{PieranskiGordian} proposed numerical candidates for a thick Gordian
unknot.  Coward and Hass \cite{CowardHass} proved the existence of a
two-component Gordian split link when the length of each component is
fixed.  Kusner and Kusner \cite{KusnerKusner} constructed a pair of
isotopic links that cannot be joined while preserving total length and
thickness; their public construction gives two global minimizers in the
same link homotopy class.  Such separation of global minima is conceptually
different from a component born above the global minimum.
Ayala and Hass \cite{AyalaHass} construct Gordian unlinks with prescribed
component lengths.  Bauermeister \cite{Bauermeister} allows length trading
between components and constructs a Gordian split link, as well as
non-global local minima and sinks in the Gehring problem for two-component
link homotopy classes.  A sink cannot be shortened by a path in its own
sublevel; this is the continuous notion closest to our strong local-minimum
components.  These link results alone do not establish a standard-BFACF Gordian
unknot.  A concrete BFACF-frozen unknot and an unbounded sequence of such
births are constructed in \cref{subsec:frozen-unlinks}; the implication
runs from a separately verified finite lattice certificate and a splice
argument, not from the continuum Gordian constructions.

Ayala's thin models impose an independent curvature bound and tube-width
constraint \cite{AyalaThin,AyalaConstraints}; the latter preprint states
a Gordian-unknot construction in that setting.  We do not identify that
model with the usual thick-rope model or with lattice edge number.
Bauermeister's discussion \cite[Remark~4.15]{Bauermeister} retains the
one-component thickly embedded Gordian-pair problem as open.
Any transfer of a continuum obstruction to BFACF must control length,
thickness and, where imposed, curvature throughout each realized move,
as required by \cref{sec:controlled-comparison}.

There are also directly relevant discrete precedents.  Mann et al.
\cite[Section~3.2.5]{MannMinimalKnots} note that their local reduction
algorithm does not simplify every unknot, referring to tangled curves
studied by Kusner and Sullivan \cite{KusnerSullivanEnergies}.  Failure of
that algorithm is not a closure certificate for all neutral BFACF paths.
Ishihara et al.\ \cite{IshiharaSlabs} prove that a trefoil confined to a
width-one slab needs at least $26$ edges, and report distinct confined
BFACF classes; the infinite-component hypothesis used later is stronger.
By contrast, Dynnikov \cite{Dynnikov} proves monotonic simplification of
unknot arc-presentations.  Its complexity is arc number and its elementary
moves are not BFACF, so that theorem does not imply monotonicity of lattice
edge number.  These comparisons motivate precise problems in
\cref{sec:further}, without claiming that a literature search proves
novelty of the lattice existence questions.

\subsection*{Main contributions}

The contributions of this paper are the following.

\begin{enumerate}[label=(\arabic*)]
\item We formulate a size-filtered local-move framework for knot
representatives and its seed-generated finite subgraphs, placing standard move
graphs inside a quantitative filtration.
\item We isolate the corresponding merge level and the equivalent minimax
reconfiguration height.  The associated ultrapseudometric is recorded as the
standard hierarchical structure forced by nested component mergers, rather than
as a phenomenon specific to lattice knots.
\item We specialize to the standard simple-cubic BFACF system, explicitly using
its classical ergodicity on fixed knot type \cite{BFACF}, and distinguish this
global fact from bounded-cap connectivity.
\item We relate the framework to earlier minimal-lattice enumerations and type-0
BFACF class decompositions \cite{Scharein2009,MinimalKnotsRechnitzer}; the new
computational target is the first higher length level at which prescribed
minimal-layer classes merge under the full BFACF move set.
\item For general local-move systems we state a PL-realizability criterion
under which move paths give finite diagrammatic certificates; for the standard
BFACF moves used here, the required local isotopy property is classical
\cite{BFACF}.
\item We report reproducible seed-generated computations for \(3_1\), \(4_1\),
\(4_1!\), \(6_3\), \(6_3!\), \(5_1\), \(5_2\), and composite trefoil examples.
\item We determine the BFACF merge tree for the published simple-cubic $4_1$
length-$30$ census: four type-$0$ components of sizes $58,58,18,18$ at
$N=30$ merge simultaneously into one component at $N=32$.
\item We determine the corresponding merge tree for the published $6_3$
length-$40$ census: twelve type-$0$ components of sizes
$39,39,10,10,9,9,7,7,6,6,3,3$ at $N=40$ merge into two mirror-related
components at $N=42$, which merge at $N=44$.
\item For one trefoil chirality, we certify six strong local-minimum
components at length $26$, their exact merge heights $28$ and $30$ with
the minimal layer, and the persistence of the resulting nine selected
branches.  We state separately the population hypothesis needed to conclude
that these exhaust the births through length $26$.
\item We characterize strong local minima as component births and distinguish
them from polygons with no immediate shortening move.  We construct a
$64$-edge BFACF-frozen Hamiltonian unknot, independently check its length-$66$
reduction to a square, and obtain an explicit $64m$-edge frozen unknot for
every $m\geq1$ by a planar two-edge block splice.  In particular, the
unknot has infinitely many nonminimal strong birth components.
\item An explicit protected-port connected sum of the $64$-edge frozen
unknot and an isolated minimal $24$-edge trefoil produces an $88$-edge
frozen trefoil.  Iterating the Hamiltonian block yields frozen trefoils
of length $64m+24$ for all $m\ge1$, and reflection covers both
chiralities.  Thus the trefoil, as well as the unknot, has infinitely
many nonminimal strong birth components.  A complete search of the
$2448$ port-aligned joins of the $75$ minimal trefoil classes to the
block shows that exactly four are frozen, all using the isolated summand,
so the choice of a protected port is essential.
\item For an arbitrary simple-cubic polygon $Q$ with $m$ edges, we
construct a BFACF-frozen polygon $\Sigma(Q)$ of the same oriented knot
type and length $125m$.  Each vertex of $Q$ is replaced by one of two
$125$-point Hamiltonian block paths, used in all orientations, with ports
at face centres.  The proof is symbolic apart from the finite check of
the two base blocks.  Hence every knot type has infinitely many
nonminimal strong birth components and lattice Gordian representatives.
Its capped BFACF graphs are not eventually connected
(\cref{thm:frozen-every-knot,cor:every-knot}).
\item As elementary consequences of finite sublevels and classical BFACF
connectivity, we characterize eventual connectivity by bounded strong birth
levels, and relate nonminimal strong minima to a positive barrier for
reaching the minimal layer.  The explicit frozen-unknot and frozen-trefoil
families show that eventual connectivity fails for $K=0_1$ and both
chiralities of $3_1$, and the block substitution extends this to every
knot type, without conflicting with classical unrestricted
BFACF ergodicity.
\item We retain explicit seed-to-mirror certificates realizing the extremal
barriers $2$ for $4_1$ and $4$ for $6_3$, and archive independently checkable
move paths for both.
\item We verify the knot types of the original prototype seeds independently through
Alexander polynomials computed from generic projections, and we confirm the
\(6_3\) merge scale with a second, independently constructed minimal seed and
an independent search implementation.
\item We draw a coordinate-level trefoil certificate with optimal peak
$30$, and verify singleton cap-$48$ components for published $8_3$ and
$8_7$ representatives.  An explicit order-four rotoreflection of the
$8_3$ representative gives a zero mirror barrier, providing a contrast
with the positive barriers of $4_1$ and $6_3$.
\item We state experimental problems suggested by the data, including the
distribution and possible growth of BFACF mirror barriers over minimal
representatives and across knot types.
\end{enumerate}

\section{Abstract filtered local-move models}
\label{sec:abstract}

We begin with an abstract formulation.  This keeps the main results independent
of the choice of lattice or move system.

\begin{definition}[Filtered local-move model]
Let \(K\) be a knot type.  A \emph{filtered local-move model} for \(K\) is a
triple
\[
\mathfrak D(K)=(\mathcal S(K),\mathcal M,\ell),
\]
where:
\begin{enumerate}[label=(\roman*)]
\item \(\mathcal S(K)\) is a countable set of finite representatives of \(K\);
\item \(\mathcal M\) is a symmetric set of local moves between elements of
\(\mathcal S(K)\), each move preserving the knot type \(K\);
\item \(\ell:\mathcal S(K)\to \mathbb N\) is a complexity function;
\item for every \(N\), the set
\[
\mathcal S_N(K)=\{x\in \mathcal S(K)\mid \ell(x)\leq N\}
\]
is finite.
\end{enumerate}
\end{definition}

The associated graph is defined as follows.

\begin{definition}[Filtered move graph]
For a filtered local-move model \(\mathfrak D(K)\), define the level-\(N\) graph
\[
G_N^{\mathfrak D}(K)
\]
by
\[
V(G_N^{\mathfrak D}(K))=\mathcal S_N(K),
\]
with an edge between \(x\) and \(y\) if and only if \(x\) and \(y\) differ by one
allowed move in \(\mathcal M\).
\end{definition}

\begin{definition}[Seed-generated component and explored subgraph]
Let \(x\in\mathcal S_N(K)\).  The \emph{seed-generated component} at level \(N\),
denoted
\[
 C_N^{\seed}(x)\subset G_N^{\mathfrak D}(K),
\]
is the connected component of \(x\) in \(G_N^{\mathfrak D}(K)\).  In an actual
computation one may stop after a vertex or time cap.  The resulting explored
subgraph will be denoted by
\[
 E_N^{\seed}(x)\subset C_N^{\seed}(x).
\]
Thus an exhaustive computation gives \(E_N^{\seed}(x)=C_N^{\seed}(x)\), whereas a
capped computation records only a certified finite part of the component.
\end{definition}

\begin{definition}[Seed-generated merge scale]
Let \(x,y\) be two representatives at the first nonempty filtration
level (the initial layer is formalized in the following section).  The
\emph{seed-generated merge scale} \(m_{\seed}(x,y)\) is the first level \(N\) such
that \(x\) and \(y\) lie in the same component of \(G_N^{\mathfrak D}(K)\).  A
positive search certificate gives an upper bound, and an exhaustive separation
certificate at a level gives a lower bound.  If a capped search finds no path, the
outcome is recorded as inconclusive.
\end{definition}

Seed-generated data are weaker than complete computations of
\(G_N^{\mathfrak D}(K)\), but they are often the computationally accessible
objects.  Throughout the experimental sections below, all capped searches are
identified explicitly.

Since \(\mathcal S_N(K)\subset \mathcal S_{N+1}(K)\), the graphs form a nested
sequence of induced subgraphs:
\[
G_N^{\mathfrak D}(K)\subset G_{N+1}^{\mathfrak D}(K)\subset G_{N+2}^{\mathfrak D}(K)
\subset\cdots.
\]

\begin{definition}[Discrete admissible components]
The \emph{discrete admissible components} of \(K\) at level \(N\) in the model
\(\mathfrak D\) are the connected components of \(G_N^{\mathfrak D}(K)\):
\[
\pi_0^{\mathfrak D}(K;N)=\pi_0(G_N^{\mathfrak D}(K)).
\]
\end{definition}

\begin{proposition}[Component persistence]
The inclusions
\[
G_N^{\mathfrak D}(K)\subset G_{N+1}^{\mathfrak D}(K)
\]
induce maps
\[
\pi_0(G_N^{\mathfrak D}(K))\longrightarrow \pi_0(G_{N+1}^{\mathfrak D}(K)).
\]
Thus the discrete admissible components form a one-parameter persistence object.
\end{proposition}

\begin{proof}
Every connected component of \(G_N^{\mathfrak D}(K)\) is contained in a connected
component of \(G_{N+1}^{\mathfrak D}(K)\), because the former graph is a subgraph
of the latter.  This gives the asserted map on connected components.  Functoriality
for successive inclusions is immediate.
\end{proof}

\section{Initial layers and merge scales}
\label{sec:initial}

The analogue of the ropelength birth level is the minimal complexity level.

\begin{definition}[Initial level and discrete ideal layer]
Let
\[
N_0^{\mathfrak D}(K)=\min\{N\in\mathbb N\mid \mathcal S_N(K)\neq\varnothing\}.
\]
The \emph{discrete ideal layer} is
\[
I_{\mathfrak D}(K)=\mathcal S_{N_0^{\mathfrak D}(K)}(K).
\]
The corresponding initial component set is
\[
\mathcal C_0^{\mathfrak D}(K)
=
\pi_0(G_{N_0^{\mathfrak D}(K)}^{\mathfrak D}(K)).
\]
\end{definition}

\begin{definition}[Discrete merge scale]
Let \(C,D\in \mathcal C_0^{\mathfrak D}(K)\).  The \emph{discrete merge scale} of
\(C\) and \(D\) is
\[
m_{\mathfrak D}(C,D)=
\min\{N\geq N_0^{\mathfrak D}(K)
\mid C\text{ and }D\text{ have the same image in }\pi_0(G_N^{\mathfrak D}(K))\},
\]
if such an \(N\) exists, and \(m_{\mathfrak D}(C,D)=\infty\) otherwise.
\end{definition}

In a globally complete move system every pair of initial components eventually
merges.  For the standard simple-cubic BFACF system this follows from the classical
ergodicity theorem \cite{BFACF}.  Allowing the value \(\infty\) is nevertheless
convenient in the abstract framework and for other move systems.

\begin{theorem}[Merge ultrapseudometric]
Assume that \(m_{\mathfrak D}(C,D)<\infty\) for all
\(C,D\in\mathcal C_0^{\mathfrak D}(K)\).  Then
\[
d_{\mathfrak D}(C,D)=m_{\mathfrak D}(C,D)-N_0^{\mathfrak D}(K)
\]
defines an ultrapseudometric on \(\mathcal C_0^{\mathfrak D}(K)\); that is,
\[
d_{\mathfrak D}(C,C)=0,
\]
\[
d_{\mathfrak D}(C,D)=d_{\mathfrak D}(D,C),
\]
and
\[
d_{\mathfrak D}(C,E)\leq
\max\{d_{\mathfrak D}(C,D),d_{\mathfrak D}(D,E)\}.
\]
\end{theorem}

\begin{proof}
Reflexivity and symmetry are immediate from the definition.  For the strong
triangle inequality, let
\[
A=m_{\mathfrak D}(C,D),\qquad B=m_{\mathfrak D}(D,E),
\]
and put \(M=\max\{A,B\}\).  At level \(M\), the images of \(C\) and \(D\) lie in
the same component, and the images of \(D\) and \(E\) lie in the same component.
Hence the images of \(C\) and \(E\) lie in the same component of
\(G_M^{\mathfrak D}(K)\).  Therefore
\[
m_{\mathfrak D}(C,E)\leq M.
\]
Subtracting \(N_0^{\mathfrak D}(K)\) gives the desired inequality.
\end{proof}

\begin{remark}
The proof uses only the nested nature of the filtered graphs.  Thus the
ultrapseudometric is not special to lattice knots; it is the familiar hierarchical
geometry associated with a nested family of connected-component partitions, as
in degree-zero persistence and merge-tree constructions; see, for example,
\cite{EdelsbrunnerHarer}.  The point here is not the abstract inequality itself,
but the knot-theoretic interpretation and the exact computation of the relevant
merge heights.
\end{remark}

\section{The lattice-filtered model}
\label{sec:lattice}

We now specialize the abstract framework to lattice knots.

\subsection{Lattice polygons}

Let \(\mathcal L\) be a three-dimensional lattice, such as the simple cubic lattice
\(\mathbb Z^3\), the face-centered cubic lattice, or the body-centered cubic
lattice.  A \emph{lattice polygon} is a simple closed polygonal curve whose edges
are lattice edges.  Its \emph{lattice length} \(\ell(\omega)\) is the number of
lattice edges in \(\omega\).

For a knot type \(K\), let
\[
\mathcal P^{\mathcal L}(K)
=
\{\omega\subset\mathcal L\mid \omega\text{ is a lattice polygon of knot type }K\}/\Isomp(\mathcal L),
\]
where \(\Isomp(\mathcal L)\) denotes the orientation-preserving lattice
isometry group.  We never identify spatial reflections in the computations;
the full-isometry quotient is a different optional model.  For
\(N\in\mathbb N\), put
\[
\mathcal P_N^{\mathcal L}(K)=
\{\omega\in\mathcal P^{\mathcal L}(K)\mid \ell(\omega)\leq N\}.
\]

\begin{lemma}[Finiteness]
For each \(N\), the set \(\mathcal P_N^{\mathcal L}(K)\) is finite.
\end{lemma}

\begin{proof}
A lattice polygon of length at most \(N\) has at most \(N\) edges.  After
quotienting by translations, one may place one vertex at the origin.  There are
only finitely many edge-direction sequences of length at most \(N\) in a fixed
lattice, and only finitely many of these close up and give self-avoiding
polygons.  Quotienting further by the finite point group of the lattice preserves
finiteness.
\end{proof}

\subsection{Local move systems}

There are several natural choices of local moves.

\begin{example}[BFACF moves]
In the simple cubic lattice, the standard BFACF moves are reversible local
modifications of self-avoiding lattice polygons.  They consist of neutral moves
and moves changing the length by \(\pm2\).  Each move preserves knot type, and a
classical theorem of Janse van Rensburg and Whittington states that the
irreducibility classes of the BFACF moves on unrooted simple-cubic polygons are
exactly the knot types \cite{BFACF}.  Thus any two simple-cubic polygons of the
same knot type are connected by a finite BFACF path when no length cap is imposed.
\end{example}

\begin{remark}[Global ergodicity versus bounded connectivity]
The BFACF ergodicity theorem is global and does not answer the filtered question
studied here.  A connecting BFACF path may have to pass through polygons longer
than both endpoints.  Consequently, for a fixed cap \(N\), the graph
\(G_N^{\SC,\BFACF}(K)\) can have several components even though their union is
connected in the unbounded BFACF state space.  Likewise, exploring from a finite
seed set under the cap \(N\) computes only the components reachable from those
seeds; it does not enumerate every vertex of the full bounded graph unless an
external exhaustive enumeration is supplied.  These are the two distinctions
needed in the experiments below.
\end{remark}

\begin{example}[Cubulated moves]
For cubic knots, cubulated moves form a discrete analogue of Reidemeister moves.
In this setting, isotopy of cubic knots can be characterized by the existence of
a finite sequence of cubulated moves \cite{CubulatedMoves}.  This example is
included to emphasize that completeness is a property of a chosen move system and
must be cited or proved separately from knot-type preservation of individual
moves.
\end{example}

\begin{definition}[Lattice-filtered move graph]
Fix a lattice \(\mathcal L\) and a local move system \(\mathcal M_{\mathcal L}\).
The \emph{lattice-filtered move graph} of \(K\) at level \(N\) is the graph
\[
G_N^{\mathcal L,\mathcal M}(K)
\]
with vertex set \(\mathcal P_N^{\mathcal L}(K)\) and with an edge between two
vertices if they differ by one move in \(\mathcal M_{\mathcal L}\).
\end{definition}

By the finiteness lemma, \(G_N^{\mathcal L,\mathcal M}(K)\) is a finite graph.
Thus its connected components and merge scales are computable by finite graph
algorithms, at least in principle.

\begin{definition}[Lattice admissible components]
The \emph{lattice admissible components} of \(K\) at lattice length \(N\) are
\[
\pi_0^{\mathcal L,\mathcal M}(K;N)
=
\pi_0(G_N^{\mathcal L,\mathcal M}(K)).
\]
\end{definition}

\begin{definition}[Minimal lattice layer]
Let
\[
n_{\mathcal L}(K)=
\min\{N\mid \mathcal P_N^{\mathcal L}(K)\neq\varnothing\}.
\]
The \emph{minimal lattice layer} is
\[
I_{\mathcal L}(K)=\mathcal P_{n_{\mathcal L}(K)}^{\mathcal L}(K).
\]
This is the lattice analogue of the ropelength ideal stratum.
\end{definition}

\begin{definition}[Lattice merge scale]
For two components
\[
C,D\in \pi_0(G_{n_{\mathcal L}(K)}^{\mathcal L,\mathcal M}(K)),
\]
define
\[
m_{\mathcal L,\mathcal M}(C,D)=
\min\{N\geq n_{\mathcal L}(K)
\mid C,D\text{ have the same image in }\pi_0(G_N^{\mathcal L,\mathcal M}(K))\},
\]
if such an \(N\) exists.
\end{definition}

\begin{corollary}[Lattice merge ultrapseudometric]
If every pair of initial lattice components eventually merges, then
\[
d_{\mathcal L,\mathcal M}(C,D)
=
m_{\mathcal L,\mathcal M}(C,D)-n_{\mathcal L}(K)
\]
defines an ultrapseudometric on the set of components of the minimal lattice
layer.
\end{corollary}

\begin{proof}
This is a direct application of Theorem 3.3 to the filtered local-move model
\(\mathfrak D=(\mathcal P^{\mathcal L}(K),\mathcal M_{\mathcal L},\ell)\).
\end{proof}

\subsection{Nonminimal births and local minima}
\label{subsec:local-minima}

For the rest of this subsection let $G_N=G_N^{\SC,\BFACF}(K)$, with
orientation-preserving lattice isometries as the quotient.  Simple-cubic
polygon lengths are even.  The following distinction is essential when
extending a minimal-layer merge tree to higher births.

\begin{definition}[Immediate and strong local minima]
A polygon $P$ is \emph{immediately nonshortenable} if it admits no legal
BFACF move of length change $-2$.  A \emph{strong local-minimum component}
at length $n$ is a connected component $D$ of the type-$0$ graph on
length-$n$ polygons such that no vertex of $D$ admits a shortening move.
Equivalently, none of its representatives can be shortened after any
finite sequence of type-$0$ moves.
\end{definition}

\begin{proposition}[Birth criterion]
\label{prop:birth-criterion}
A type-$0$ component $D$ at length $n$ is a strong local-minimum component
if and only if it is a connected component of $G_n$ disjoint from $G_{n-2}$.
Consequently, such components are exactly the branches born at level $n$
in component persistence.
\end{proposition}
\begin{proof}
If $D$ has no shortening move, every move from $D$ allowed under cap $n$
is neutral and remains in $D$.  Thus $D$ is a whole component of $G_n$.
Conversely, a component of $G_n$ disjoint from $G_{n-2}$ consists only of
length-$n$ vertices.  Its edges are neutral, and a shortening move would
join it to $G_{n-2}$, a contradiction.
\end{proof}

\begin{proposition}[Arbitrarily long immediately nonshortenable polygons]
\label{prop:dilation-minima}
Let $P$ be any simple-cubic lattice polygon of knot type $K$.  For every
integer $m\geq2$, dilating $P$ by $m$ and subdividing into unit edges
gives an immediately nonshortenable polygon $P^{(m)}$ of the same knot type,
with $\ell(P^{(m)})=m\ell(P)$.  In particular, such lengths are unbounded
for each lattice knot type.
\end{proposition}
\begin{proof}
Dilation preserves embeddedness and knot type.  Every maximal straight
run in $P^{(m)}$ has at least $m\geq2$ edges.  A shortening BFACF move
requires three consecutive sides of a unit plaquette, whose direction
word is $a,b,-a$ with $a\perp b$.  The middle edge is a straight run
of length one, so this pattern cannot occur in $P^{(m)}$.
\end{proof}

\begin{remark}
\Cref{prop:dilation-minima} does not produce strong local minima: neutral
moves can create a shortening opportunity.  For example, the boundary
of a $2\times2$ square has no shortening move, but flipping one corner
by a neutral move creates three sides of a unit plaquette that can then
be shortened.  Thus unbounded lengths of immediately nonshortenable
polygons do not imply unbounded birth levels of the merge tree.
\end{remark}

\subsection{Gordian barriers and eventual connectivity}
\label{subsec:eventual-connectivity}

Continue with the standard simple-cubic BFACF graph and proper-cubic
quotient.  Write $n_K=n_{\SC}(K)$ and $\mathcal M_K=I_{\SC}(K)$.
The classical theorem \cite{BFACF}, also restated in
\cite[Theorem~1]{BFACFExtensions}, gives a finite minimax height
\[
 H(P,Q)=\min_{\Gamma:P\leadsto Q}\max_{R\in\Gamma}\ell(R).
\]
Define the height and excess for reaching \emph{some} minimal representative by
\[
 h_K(P)=\min_{Q\in\mathcal M_K}H(P,Q),\qquad
 e_K(P)=h_K(P)-\ell(P)\geq0.
\]
This minimum over the whole minimal layer must be distinguished from a
height to a prescribed minimal vertex.  We call $P$ \emph{lattice Gordian}
in this specified model if $e_K(P)>0$.  For the unknot, $\mathcal M_K$
is the single proper-cubic class of the unit square $\square$, so the
condition is $H(P,\square)>\ell(P)$.

\begin{proposition}[Finite-input merge bounds]
\label{prop:finite-input-bound}
Fix $Q_*\in\mathcal M_K$.  For every integer $L\geq n_K$,
\[
 M_K(L;Q_*)=\max_{P\in V(G_L(K))}H(P,Q_*)<\infty.
\]
In particular, all vertices present by level $L$ merge under some finite cap.
\end{proposition}
\begin{proof}
There are finitely many quotient vertices under cap $L$, and every
height in the maximum is finite by BFACF connectivity.
\end{proof}

\begin{proposition}[Gordian configurations and nonminimal births]
\label{prop:gordian-birth}
A lattice Gordian configuration of type $K$ exists if and only if
a strong local-minimum component is born at some $n>n_K$.
\end{proposition}
\begin{proof}
If $h_K(P)>\ell(P)=N$, the component $C$ of $P$ in $G_N(K)$ misses
$\mathcal M_K$.  Choose a vertex $Q\in C$ of least length $n>n_K$.
The entire neutral component of $Q$ lies in $C$.  No vertex of that
neutral component can have a shortening move, since the shorter neighbor
would also lie in $C$, contradicting minimality of $n$.
Conversely, a strong local-minimum component $D$ born at $n>n_K$
is a component of $G_n(K)$ disjoint from $G_{n-2}(K)$ by
\cref{prop:birth-criterion}.  It therefore misses the minimal layer,
and $h_K(P)\geq n+2$ for every $P\in D$.
\end{proof}

\begin{proposition}[Eventual connectivity]
\label{prop:eventual-connectivity}
For a fixed knot type $K$, the following are equivalent:
\begin{enumerate}[label=(\roman*)]
\item $G_N(K)$ is connected for every sufficiently large $N$;
\item the birth lengths of strong local-minimum components are bounded;
\item there are only finitely many strong local-minimum components.
\end{enumerate}
Consequently, unbounded strong birth levels are equivalent to
disconnectedness of $G_N(K)$ at arbitrarily large caps.
\end{proposition}
\begin{proof}
If all $G_N(K)$ with $N\geq N_0$ are connected, then no strong component
can be born at $n>\max\{N_0,n_K\}$: it would be disjoint from the
nonempty lower sublevel.  Thus (i) implies (ii).
Finite sublevels make (ii) and (iii) equivalent.

Suppose (ii) holds with bound $B\geq n_K$.  By
\cref{prop:finite-input-bound}, choose $M\geq B$ so that all vertices
of $G_B(K)$ lie in a single component of $G_M(K)$.
For $N\geq M$, let $C$ be any component of $G_N(K)$ and choose a
least-length vertex $P\in C$.  As in the preceding proof, the neutral
component of $P$ has no shortening move and is a strong local-minimum
component.  Hence $\ell(P)\leq B$.  Every component of $G_N(K)$ therefore
meets $G_B(K)$, whose vertices have already merged at $M$.
Thus $G_N(K)$ is connected, proving (i).
\end{proof}

\begin{corollary}[Unbounded excess requires unbounded births]
\label{cor:excess-births}
If $\sup_P e_K(P)=\infty$, then the strong birth levels for $K$ are unbounded.
\end{corollary}
\begin{proof}
Otherwise \cref{prop:eventual-connectivity} gives $N_0$ beyond which
all capped graphs are connected.  Then $e_K(P)=0$ for
$\ell(P)\geq N_0$, and the remaining finitely many vertices have bounded
excess.
\end{proof}

The converse is not a formal consequence of component persistence:
an abstract filtered graph can have infinitely many births, each merging
after just one two-unit increase.  Thus birth height and escape excess
are separate quantitative questions.  The preceding propositions are
elementary consequences of finiteness and classical connectivity, recorded
to specify the quantifiers in those questions.

\section{Scope of the discrete model}
\label{sec:scope}
The resource in this paper is lattice edge number. The continuum and
projection-framed theories use distinct geometric state spaces and resources.
The local-move model, its merge partitions, and finite certificates are
specified here independently. A comparison requires explicit controlled maps;
a formal correspondence diagram is not such a construction.

\section{Discrete paths as finite certificates}
\label{sec:certificates}

A key advantage of the lattice model is that an explicitly realized discrete path
can be converted into finite diagrammatic data.  The following formulation makes
the required hypothesis on the local move system explicit.

\begin{definition}[PL-realizable local move system]
A local move system \(\mathcal M_{\mathcal L}\) is called \emph{PL-realizable} if
every allowed local move \(\omega\leadsto\omega'\) admits an ambient PL isotopy
\(H_t\) of \(\mathbb R^3\), supported in a ball meeting the polygon only in the
local move region, such that \(H_0(\omega)=\omega\), \(H_1(\omega)=\omega'\), and
\(H_t(\omega)\) is an embedded polygonal knot for all \(t\).  This is an
additional hypothesis on the move system, not a formal consequence of
self-avoidance of the endpoints.
\end{definition}

\begin{remark}[BFACF and cubulated moves]
The cubulated move system of Hinojosa--Verjovsky--Verjovsky Marcotte
\cite{CubulatedMoves} provides a natural example of a PL-realizable local move
system in the sense of the preceding definition.  Standard BFACF moves are also
local polygonal deformations realized by ambient isotopy; the proof of BFACF
ergodicity is built from such local isotopies \cite{BFACF}.  Thus the explicit
BFACF certificates below may be interpreted as finite PL-isotopy certificates.
For an arbitrary implementation or a different move system, however, the
PL-realizability hypothesis should still be checked against the precise move
rules being used.
\end{remark}

\begin{theorem}[Discrete-to-finite projection]\label{thm:discrete-to-finite}
Let \(\mathcal M_{\mathcal L}\) be PL-realizable, and let
\(\omega_0,\omega_1,\ldots,\omega_r\) be a path in
\(G_N^{\mathcal L,\mathcal M}(K)\).  For a generic projection direction \(u\), the
projected diagrams
\[
D_u(\omega_0),D_u(\omega_1),\ldots,D_u(\omega_r)
\]
are related by a finite sequence of Reidemeister moves after subdividing each
local move into a generic one-parameter PL isotopy.  Consequently, every realized
discrete admissible path gives a finite Reidemeister certificate.
\end{theorem}

\begin{proof}
For each edge \(\omega_i\leadsto\omega_{i+1}\) choose the PL ambient isotopy
provided by the PL-realizability hypothesis.  Concatenating these finitely many
isotopies gives a PL one-parameter family from \(\omega_0\) to \(\omega_r\).  After
subdividing the time interval, this family is linear on finitely many parameter
subintervals and has only finitely many vertices and edges at each time.

Choose a projection direction outside the finite union of exceptional directions
where an edge is projected to a point or where two non-adjacent edges have a
non-isolated projected overlap.  A further arbitrarily small generic perturbation
of the PL isotopy, fixed at the endpoints of the move intervals, makes the
projected one-parameter family generic in the usual Reidemeister sense: away from
finitely many times the diagram is regular, and at each singular time exactly one
standard local event occurs.  By the PL version of Reidemeister's theorem
\cite{Reidemeister,BurdeZieschang}, the change in the diagram over each move
interval is therefore a finite sequence of Reidemeister moves.  Since the lattice
path has finitely many edges, the concatenated projected change is finite.
\end{proof}

\begin{corollary}[Discrete merge certificates]
Suppose two initial lattice components \(C,D\) merge at level \(N\), that is,
\[
m_{\mathcal L,\mathcal M}(C,D)=N,
\]
and suppose the chosen local move system is PL-realizable.  Then, for a generic
projection direction \(u\), any chosen lattice path realizing the merge projects
to a finite Reidemeister certificate.
\end{corollary}

\begin{proof}
A merge at level \(N\) is represented by a finite path in
\(G_N^{\mathcal L,\mathcal M}(K)\) connecting a representative of \(C\) to a
representative of \(D\).  Apply Theorem~\ref{thm:discrete-to-finite} to this path.
\end{proof}

\begin{remark}
One may seek uniform bounds of the form
\[
M(N)\leq C N^2
\]
for diagrammatic crossing complexity, since a polygon with \(N\) edges has at
most quadratically many pairs of edges.  For move paths, however, the number of
Reidemeister moves also depends on the length of the chosen path and on the
projection direction.  Thus the most robust statement is the existence of a
finite certificate for a realized path, not a sharp universal complexity bound.
\end{remark}

\section{From finite data back to discrete models}
\label{sec:realization}

The reverse direction is possible but should be formulated as a realization
principle rather than as an equivalence of theories.

\begin{definition}[Finite-to-discrete realization property]
A finite diagrammatic model is said to have the \emph{finite-to-discrete
realization property} with respect to a lattice model
\(\mathfrak D=(\mathcal P^{\mathcal L}(K),\mathcal M_{\mathcal L},\ell)\) if the
following hold.
\begin{enumerate}[label=(\roman*)]
\item Every diagram in the finite model can be realized by a lattice polygon.
\item Every Reidemeister move in the finite model can be realized by a finite
sequence of lattice local moves in the chosen move system.
\item There is a complexity-control function \(Q\) such that diagrams of finite
complexity at most \(M\) are realized by lattice polygons of length at most
\(Q(M)\).
\end{enumerate}
\end{definition}

\begin{proposition}[Conditional finite-to-discrete lifting]
Assume the finite-to-discrete realization property.  Then any finite
Reidemeister path of diagrammatic complexity at most \(M\) lifts, relative to the
chosen realization scheme, to a path in
\(G_{Q'(M)}^{\mathcal L,\mathcal M}(K)\) for some complexity-control function
\(Q'\).
\end{proposition}

\begin{proof}
Realize each diagram in the path by a lattice polygon using condition (i).
Realize each Reidemeister move by a finite sequence of lattice local moves using
condition (ii).  Taking the maximum lattice length over the finitely many
intermediate lattice polygons gives a bound.  The assumed complexity control
provides a function \(Q'\) depending only on the chosen finite complexity bound
and the realization scheme.
\end{proof}

\begin{remark}
This statement is deliberately conditional.  It avoids asserting that finite knot
theory and discrete knot theory are equivalent.  In particular, condition (ii) is
a substantive assumption for any specific local move system.  We do not claim
here that arbitrary Reidemeister moves lift to BFACF paths with controlled
length.  The correct interpretation is that finite diagrammatic data can often be
implemented in sufficiently fine or sufficiently large lattice models, but the
construction requires explicit realization and complexity control.
\end{remark}

\section{Relation to ropelength-filtered spaces}
\label{sec:ropelength}

We now indicate how the lattice-filtered model relates to the original
ropelength-filtered space.

\subsection{Lattice polygons as thick representatives}

A raw lattice polygon is not a smooth curve and has corners.  To compare it with
ropelength theory, one may round the corners and thicken the polygon.  To avoid
collisions under smoothing, it is useful to impose a discrete thickness condition.

\begin{definition}[Combinatorially thick lattice polygon]
Let \(\rho>0\).  A lattice polygon \(\omega\) is called \emph{\(\rho\)-thick} if
there are pairwise disjoint closed regular neighborhoods of radius \(\rho\) for
all non-adjacent edges and pairwise disjoint corner balls in which the incident
edge neighborhoods meet only in the prescribed adjacent way.  For the simple cubic lattice, this condition is automatic for all self-avoiding
polygons if a fixed sufficiently small $\rho>0$ is chosen: nonincident
unit lattice edges are at distance at least $1$, and distinct lattice
vertices are at distance at least $1$.  One may take disjoint corner
balls of radius, say, $2\rho$ and shortened edge neighborhoods of radius
$\rho$, for $0<\rho<1/4$.  The condition remains explicit for other
lattices and for a prescribed larger thickness scale.
\end{definition}

\begin{proposition}[Smoothing estimate]
Fix \(\rho>0\) and a smoothing convention inside each allowed corner ball.  There
is a constant \(A_\rho>0\) such that every \(\rho\)-thick lattice polygon
\(\omega\) of lattice length \(N\) admits a smooth representative
\(\gamma_\omega\) of the same knot type with
\[
\Thi(\gamma_\omega)=1,
\qquad
\Len(\gamma_\omega)\leq A_\rho N.
\]
\end{proposition}

\begin{proof}
Inside each corner ball replace the two incident straight subarcs by the fixed
rounded arc prescribed by the smoothing convention, and shorten the adjacent
straight portions accordingly.  The \(\rho\)-thick hypothesis gives disjoint
control neighborhoods for non-adjacent edges and for the corner balls, so these
local replacements cannot create a new intersection between non-adjacent parts of
the curve.  Hence the rounded curve is ambient isotopic to the original lattice
polygon.  The same neighborhood control gives a positive lower bound, depending
only on \(\rho\) and the smoothing convention, for the reach of the rounded curve.
After scaling so that the thickness is \(1\), the length is multiplied by a
constant depending only on that lower bound.  Since the original polygon has
\(N\) unit edges and the smoothing changes length by a bounded amount per corner,
the final length is at most \(A_\rho N\).
\end{proof}

In particular, for the simple cubic lattice and a fixed sufficiently small
$\rho$, every self-avoiding lattice polygon is eligible for this rounding.
Thus there is a coarse comparison map on the \(\rho\)-thick subfamily,
\[
\mathcal P_N^{\mathcal L,\rho}(K)
\longrightarrow
Y_{A_\rho N}(K),
\]
where \(\mathcal P_N^{\mathcal L,\rho}(K)\) denotes the set of \(\rho\)-thick
lattice polygons of length at most \(N\).  For an arbitrary lattice or a prescribed larger $\rho$, the map is
restricted to the subfamily satisfying the stated separation conditions.
This is a map on representatives, not a homotopy equivalence or a
BFACF-to-ropelength merge-height comparison.

\subsection{What should not be claimed}

The preceding comparison does not imply that
\[
G_N^{\mathcal L,\mathcal M}(K)
\]
is equivalent to
\[
Y_\Lambda(K).
\]
There are several reasons:
\begin{enumerate}[label=(\roman*)]
\item lattice polygons form only a discrete subset of all thick representatives;
\item local lattice moves do not encode all continuous deformations at a fixed scale;
\item the relation between lattice length and ropelength is coarse and depends on the smoothing convention;
\item merge scales in the lattice model may depend on the lattice and move system.
\end{enumerate}

Therefore the lattice model is best regarded as a computable surrogate rather
than a canonical discretization.

\subsection{A convergence problem}

One may nevertheless ask whether finer and finer lattice models approximate the
continuous theory.

\begin{problem}[Convergence of discrete merge scales]
Let \(a>0\) be a mesh size and consider lattice polygons in \(a\mathbb Z^3\) with
a suitable discrete thickness condition.  Define lattice merge scales after
using physical edge length \(aN\) and an explicitly specified rounding and thickness convention.  Under what conditions do the resulting discrete
merge scales converge, or at least give upper and lower bounds, for the continuous
ropelength merge scales as \(a\to 0\)?
\end{problem}

This problem is central if one wants to promote discrete knot theory from a
computable surrogate to an approximation theory for ropelength-filtered knot
spaces.  The difficulty is twofold.  First, the vertex sets must approximate
thick smooth knots while retaining a uniform discrete thickness condition.
Second, the move system must approximate controlled ambient isotopies without
introducing artificial barriers or shortcuts.  Possible approaches include using
thickness-preserving cubical moves, proving rounding-and-straightening estimates
with explicit reach control, or comparing upper and lower merge barriers through
coarse Lipschitz maps between the lattice and ropelength filtrations.

\section{Computability}
\label{sec:computability}

At each fixed level, the discrete theory is finite.

\begin{theorem}[Finite computability]
Let \(\mathfrak D(K)=(\mathcal S(K),\mathcal M,\ell)\) be a filtered local-move
model.  For each \(N\), the following data are computable by finite graph
procedures, provided the finite set \(\mathcal S_N(K)\) and the move relation are
effectively enumerated:
\begin{enumerate}[label=(\roman*)]
\item the graph \(G_N^{\mathfrak D}(K)\);
\item its connected components \(\pi_0(G_N^{\mathfrak D}(K))\);
\item the image of each component under the inclusion to higher levels;
\item the merge scale of any two initial components, up to any prescribed search bound.
\end{enumerate}
\end{theorem}

\begin{proof}
The graph is finite by assumption.  Connected components are computable by
standard breadth-first or depth-first search.  The inclusion maps are obtained by
identifying the same vertices in the nested graphs.  To compute merge scales up
to a search bound \(B\), compute the components of \(G_N^{\mathfrak D}(K)\) for
\(N_0\leq N\leq B\) and record the first level at which the two initial components
have the same image.
\end{proof}

\begin{remark}
The theorem is an effective finiteness statement rather than an efficient
algorithmic complexity theorem.  Enumeration of lattice knots of length at most
\(N\) is itself difficult for large \(N\).  Nonetheless, the problem is finite
and well-defined, in contrast with the direct computation of path components of
\(Y_\Lambda(K)\).
\end{remark}

\section{A prototype implementation}
\label{sec:prototype}

The preceding computability statement is abstract.  For the purposes of
experimentation, we implemented a small Python prototype for the simple cubic
lattice.  The program is not intended as an optimized enumeration engine; rather,
it is a transparent implementation of the definitions above.  Its role is to
construct the BFACF-filtered graph from a supplied finite list of lattice
polygons, or to explore the reachable subgraph from specified seed polygons.

The prototype represents a lattice polygon as a cyclic list of vertices in
\(\mathbb Z^3\).  It canonicalizes polygons up to translation, cyclic
reparametrization, reversal of orientation, and proper cubic rotations.  If one
wishes to identify mirror images, the full cubic symmetry group may be used
instead.  It then generates three BFACF-type moves:
\begin{enumerate}[label=(\roman*)]
\item length-preserving corner flips;
\item positive edge expansions, increasing lattice length by \(2\);
\item negative edge contractions, decreasing lattice length by \(2\).
\end{enumerate}
Only moves whose outputs remain self-avoiding lattice polygons are retained.
Thus, for a supplied finite vertex set \(V\subset \mathcal P_N^{\SC}(K)\), the
program constructs the induced graph
\[
G_N^{\SC,\BFACF}(K)|_V.
\]
It then computes connected components by breadth-first search.

The following command explores the reachable BFACF graph from the square unknot
up to length \(8\):
\begin{verbatim}
python lattice_bfacf_graph.py --demo-unknot --max-len 8
\end{verbatim}
The output of this elementary test is:
\begin{verbatim}
Demo: square unknot, BFACF reachable graph up to length 8
vertices: 18
edges:    31
components: 1
component sizes: [18]
\end{verbatim}
This confirms, at the level of the prototype, that the reachable subgraph from
the square unknot remains connected through length \(8\).

For non-trivial knots, the program expects an external enumeration of lattice
polygons.  For example, if \texttt{trefoil24.json} contains a list of length
\(24\) trefoil polygons, then the induced graph at the minimal level is built by
\begin{verbatim}
python lattice_bfacf_graph.py --input trefoil24.json --max-len 24
\end{verbatim}
If \texttt{trefoil\_seed.json} contains one or more seed polygons, then the
reachable graph below length \(26\) can be explored by
\begin{verbatim}
python lattice_bfacf_graph.py --input trefoil_seed.json --max-len 26 --explore
\end{verbatim}
Thus the implementation separates two tasks.  The enumeration of all polygons of
a given knot type and length is supplied externally, while the construction of
the filtered BFACF graph, the computation of connected components, and the
calculation of merge data are performed by the program.  In the version used for
the longer braid-generated seed below, canonicalization is implemented through
cyclic edge sequences rather than by enumerating all translated cyclic vertex
lists; this gives the same canonical quotient in the examples considered here
and makes longer seed computations feasible.

\begin{remark}
This separation is important for the trefoil.  The complete enumeration of all
length \(24\) trefoils in the simple cubic lattice is a substantial lattice-knot
enumeration problem.  The present code is designed to use such data once supplied;
it does not, by itself, enumerate all minimal trefoils from scratch.
\end{remark}

For the experiments reported below, we used the input files listed in
Table~\ref{tab:input-files}.

\begin{table}[ht]
\centering
\small
\begin{tabular}{lll}
\toprule
file & role & knot type or experiment\\
\midrule
\texttt{trefoil\_seed.json} & seed polygon & $3_1$\\
\texttt{figure8\_seed.json} & seed polygon & $4_1$\\
\texttt{figure8\_mirror\_seed.json} & reflected seed polygon & $4_1!$\\
\texttt{figure8\_and\_mirror\_seeds.json} & two-seed input & $4_1$ and $4_1!$\\
\texttt{six3\_seed.json} & minimal seed polygon & $6_3$\\
\texttt{six3\_mirror\_seed.json} & reflected seed polygon & $6_3!$\\
\texttt{six3\_and\_mirror\_seeds.json} & two-seed input & $6_3$ and $6_3!$\\
\texttt{six3b\_seed.json} & independent second seed & $6_3$ (braid + BFACF descent)\\
\texttt{six3b\_mirror\_seed.json} & reflected second seed & $6_3!$\\
\texttt{six3b\_and\_mirror\_seeds.json} & two-seed input & $6_3$ second-seed check\\
\texttt{five1\_seed.json} & seed polygon & $5_1$\\
\texttt{five2\_braid\_seed.json} & stress-test seed & non-minimal $5_2$ input\\
\texttt{granny\_seed.json} & sanity-check seed & $3_1\#3_1$\\
\texttt{square\_seed.json} & sanity-check seed & $3_1\#3_1!$\\
\texttt{granny\_square\_seeds.json} & two-seed input & composite-trefoil sanity check\\
\bottomrule
\end{tabular}
\caption{Input files used in the seed-generated experiments.  The $6_3$
seed is converted from the published minimal-coordinate data cited in the text.
The \(5_2\) braid seed is a non-minimal stress-test input and is not used as
birth-level evidence for \(5_2\).}
\label{tab:input-files}
\end{table}

Unless explicitly stated otherwise, computations were run without the
\texttt{--mirror} option, so that mirror symmetries are not identified.  This
convention is essential when one wants to compare mirror-sensitive examples.

\section{Reproducible seed-generated experiments}
\label{sec:experiments}

This section reports reproducible prototype experiments together with
certified component searches and merge computations.  Completeness is
specified separately for each component, birth layer or selected collection
of branches.  The computations provide finite certificates for separation
or merging at specified length levels.
A table entry marked ``seed-generated'' is exhaustive only for the component
generated from the supplied seed under the stated length bound.  A table entry
with a cap is a partial exploration and is not used as evidence for global
connectivity.

\subsection{The unknot}

For the unknot, the minimal lattice layer consists of shortest lattice unknots.
In the simple cubic lattice these are square polygons of length \(4\), up to
lattice isometry.  The prototype computation gives a basic check: the
BFACF-reachable subgraph from the square unknot through length \(8\) has \(18\)
vertices, \(31\) edges, and one connected component.

\subsection{A frozen unknot and unbounded nonminimal births}
\label{subsec:frozen-unlinks}

The elementary eight-edge square calculation above concerns only a small
capped graph.  We now give an infinite family of \emph{new} connected
components born well above the four-edge minimum.  This is different from
the dilation examples of \cref{prop:dilation-minima}: the following polygons
have \emph{no neutral move as well as no shortening move}.  The construction
uses an explicit $64$-edge base polygon, a finite BFACF reduction certificate,
and a two-edge connected-sum splice.

Use direction digits $0,1,2,3,4,5$ for
$+x,+y,+z,-x,-y,-z$, respectively.  Let $P_1$ start at $(0,0,0)$
and follow the cyclic word
\[
\begin{gathered}
\texttt{10001133322000443243311100044453}\\[-2pt]
\texttt{51331500222345455002111334234455}.
\end{gathered}
\]
The two rows concatenate without a separator.  Direct coordinate
verification shows that this word is a self-avoiding Hamiltonian cycle on
$B=\{0,1,2,3\}^3$: it visits each of the $64$ vertices exactly once.
Its selected edges include the two disjoint boundary ports
\[
e_L=[(0,0,0),(0,0,1)],\qquad
 e_R=[(3,0,0),(3,0,1)].
\]
There is no plaquette of which three sides are consecutive polygon
edges, hence no type-$-2$ BFACF move.  At each two-edge corner, the
fourth vertex of the supporting unit plaquette belongs to $B$ and is
occupied, hence no type-$0$ move is legal either.
An explicit archived $25{,}436$-move BFACF path reduces $P_1$ to the
unit square, never exceeding $66$ edges.  Every polygon and every move
of the path is checked by the separate coordinate verifier in
\path{anc/frozen_unknots/base64/}.  In particular, this provides a finite
unknot certificate independent of Alexander-polynomial recognition.
As the first BFACF move from $P_1$ must increase length, we obtain
\[
\boxed{H(P_1,\square)=66.}
\]

\begin{figure}[ht]
\centering
\begin{minipage}{0.48\linewidth}\centering
\includegraphics[width=\linewidth]{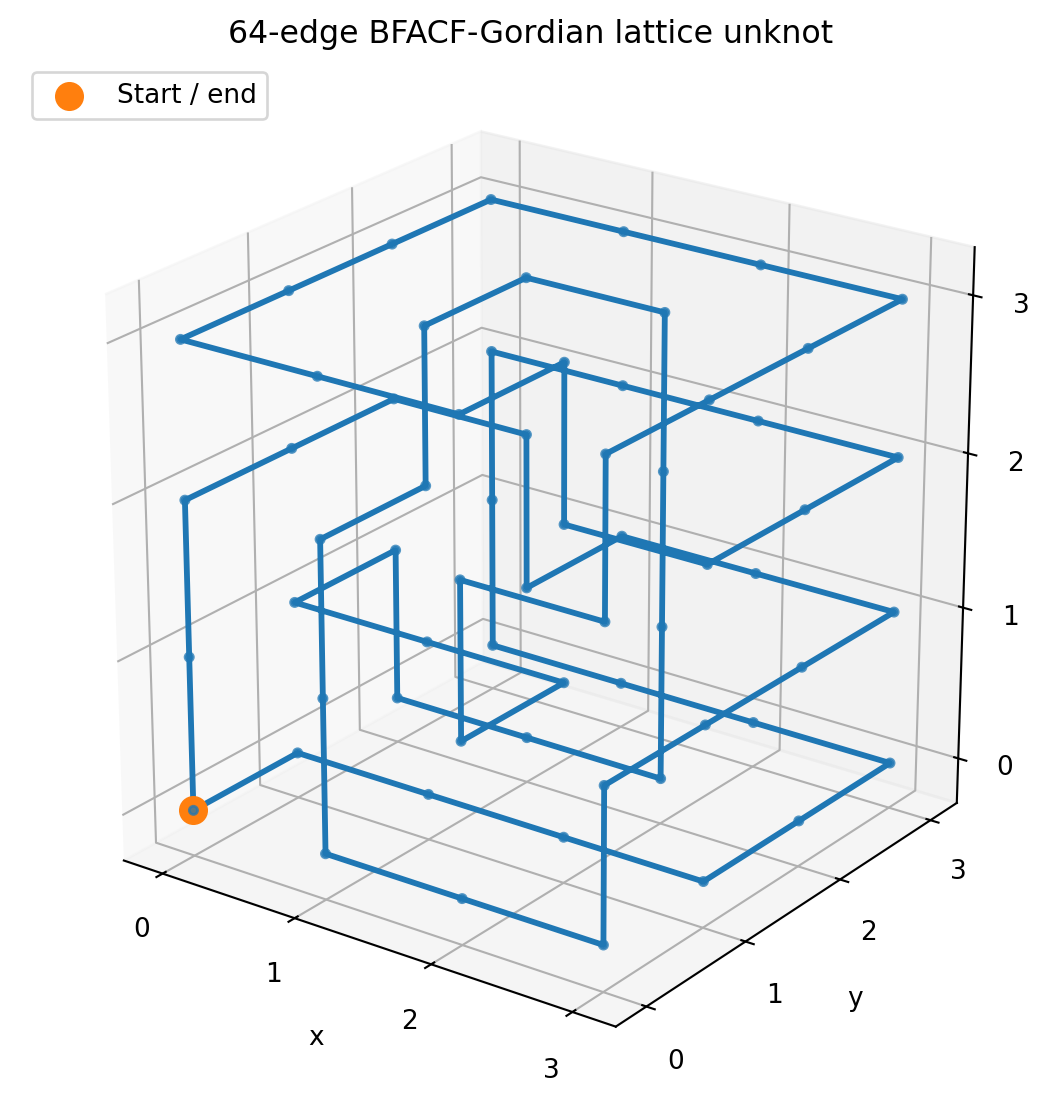}
\end{minipage}\hfill
\begin{minipage}{0.48\linewidth}\centering
\begin{tikzpicture}[scale=.73]
\fill[cyan!10] (-2.7,-.3) rectangle (-.2,2.25);
\fill[cyan!10] (.2,-.3) rectangle (2.7,2.25);
\draw[gray,thick] (-2.7,-.3) rectangle (-.2,2.25);
\draw[gray,thick] (.2,-.3) rectangle (2.7,2.25);
\draw[blue,very thick] (-.2,.5)--(.2,.5) (-.2,1.45)--(.2,1.45);
\draw[red,densely dashed,thick] (-.2,.5)--(-.2,1.45) (.2,.5)--(.2,1.45);
\node at (-1.45,1.94) {$P^{(k)}$};
\node at (1.45,1.94) {$P^{(k+1)}$};
\node[blue] at (0,-.75) {added bridges};
\node[red] at (0,2.62) {deleted port edges};
\end{tikzpicture}
\end{minipage}
\caption{Left: the $64$-edge frozen Hamiltonian unknot (the three-dimensional
picture is a projection, not a plane knot diagram).  Right: the two-edge
untwisted splice between adjacent blocks.  Dashed edges are deleted and
solid blue edges are inserted.  The other edges are unchanged.}
\label{fig:frozen-splice}
\end{figure}

\begin{theorem}[Infinitely many BFACF-frozen unknots]
\label{thm:infinitely-many-frozen-unknots}
For every integer $m\ge1$ there is a simple-cubic lattice unknot $P_m$
with
\[
\begin{aligned}
\ell(P_m)&=64m,\\
\#\{\text{legal type-$0$ moves from }P_m\}&=0,\\
\#\{\text{legal type-$-2$ moves from }P_m\}&=0.
\end{aligned}
\]
Thus $[P_m]$ is an isolated vertex of $G_{64m}(0_1)$, and
\[
H(P_m,\square)\ge64m+2.
\]
Consequently the unknot has infinitely many nonminimal strong
local-minimum components, and its capped BFACF graphs are not eventually
connected.
\end{theorem}
\begin{proof}
For $0\le k<m$, take a translated copy
$P^{(k)}=P_1+(4k,0,0)$ in
\[
B_k=\{4k,4k+1,4k+2,4k+3\}\times\{0,1,2,3\}^2.
\]
For every $0\le k<m-1$ remove the right port of $P^{(k)}$ and
the left port of $P^{(k+1)}$, namely
\[
[(4k+3,0,0),(4k+3,0,1)],\quad
[(4k+4,0,0),(4k+4,0,1)],
\]
and insert the two bridges
\[
[(4k+3,0,z),(4k+4,0,z)],\qquad z=0,1.
\]
Both ports in every intermediate block are distinct, so the surgeries
do not interfere.  Each surgery merges two disjoint cycles into one
self-avoiding cycle and uses the same four port vertices.  The resulting
single polygon $P_m$ visits every vertex of
\[
R_m=\{0,\ldots,4m-1\}\times\{0,1,2,3\}^2
\]
exactly once.  Hence $\ell(P_m)=64m$.

The separating plane $x=4k+\tfrac72$ meets the spliced cycle in two
points, at the centers of the two bridges.  Capping on either side
by the corresponding deleted port edge realizes precisely the original
block polygons.  The two bridges form a flat, untwisted band in a
unit plaquette.  Thus the surgery is the usual connected sum, and
\[
[P_m]=[P_1]^{\#m}=0_1.
\]

Every corner-swap type-$0$ move would use the fourth vertex of its
supporting unit square.  Since all vertices of the rectangular box
$R_m$ are occupied, and the fourth vertex is still in $R_m$, no such
move is legal.

It remains to rule out a type-$-2$ move, i.e. three consecutive
sides of a unit plaquette.  All noninterface plaquettes are governed
by the checked base block.  Deleting an edge never creates a new
three-edge plaquette.  Only plaquettes containing a new bridge need
inspection.  At each interface the two $x$-bridges lie in the same
$xz$-plaquette, while its two vertical $z$-edges are precisely the
deleted edges.  The remaining potentially dangerous plaquettes are
in the $xy$ plane at $z=0,1$ and the neighboring $xz$ plaquette at
$z=1,2$.  In the base block, the right port at $(3,0,0)$ has
its other edge in direction $-x$, whereas at $(3,0,1)$ its other
edge is $+y$.  The left port at $(0,0,0)$ has its other edge
$+y$, whereas at $(0,0,1)$ its other edge is $+z$.
Consequently none of these plaquettes contains more than two selected
edges.  The outer-side plaquettes have still fewer edges.
Therefore no shortening move exists.  This finite port check is
independent of $m$ and is also implemented in
\path{anc/frozen_unknots/construct_family.py}.

Every BFACF path beginning at $P_m$ must consequently start with a
$+2$ move, so its maximum length is at least $64m+2$.  Since the
unit square has shorter length, it is not in the level-$64m$
component of $P_m$; indeed this component is a singleton.  As $m$
varies, the isolated strong components are born at distinct lengths
$64m$, yielding infinitely many nonminimal births.  Finally,
\cref{prop:eventual-connectivity} implies failure of eventual
connectivity for $K=0_1$.
\end{proof}

\begin{corollary}[Persistence and failure of eventual connectivity for the unknot]
\label{cor:frozen-unknot-persistence}
For every $m\ge1$, the zero-dimensional BFACF persistence of the unknot
has a component born at length $64m$ that survives beyond that level.
Its death level $d_m$ is finite and satisfies $d_m\ge64m+2$.
In particular,
\[
\#\{\text{nonminimal strong local-minimum components of }0_1\}=\infty.
\]
Moreover, there is no $N_0$ for which every $G_N(0_1)$ with
$N\ge N_0$ is connected.
\end{corollary}
\begin{proof}
The singleton $P_m$ has no incident edges to lower or equal length,
so its component is born at $64m$.  By unrestricted BFACF ergodicity
it eventually connects to the square, which gives finite $d_m$.
The lower bound and the last assertions follow from
\cref{thm:infinitely-many-frozen-unknots,prop:eventual-connectivity}.
\end{proof}

\begin{remark}[What the construction does not claim]
\label{rem:frozen-scope}
The $64$-edge base configuration has exact barrier two,
$H(P_1,\square)-64=2$, but the exact merge levels of $P_m$ for
$m>1$ have not been computed.  The family shows unbounded \emph{birth
lengths} and a uniform positive escape-excess lower bound; it does
\emph{not} establish that
$\sup_{P\text{ unknot}}\bigl(H(P,\square)-\ell(P)\bigr)=\infty$.
It is specific to standard simple-cubic BFACF moves and does not
imply a thick-rope Gordian unknot in the continuous setting.
\end{remark}

\begin{table}[ht]
\centering
\begin{tabular}{rrrr}
\toprule
Number of blocks $m$ & Edges $64m$ & Type $0$ & Type $-2$\\
\midrule
$1$&$64$&$0$&$0$\\
$2$&$128$&$0$&$0$\\
$3$&$192$&$0$&$0$\\
$4$&$256$&$0$&$0$\\
$8$&$512$&$0$&$0$\\
$16$&$1024$&$0$&$0$\\
$32$&$2048$&$0$&$0$\\
$64$&$4096$&$0$&$0$\\
\bottomrule
\end{tabular}
\caption{Direct coordinate checks on representative members of the
Hamiltonian-block unknot family.  All cases have one closed cycle,
$64m$ occupied lattice vertices and exactly $m-1$ interfaces.  The
universal statement follows from the local splice proof, not from
these finitely many tests.}
\label{tab:frozen-family-audit}
\end{table}

\subsection{Frozen connected sums: infinitely many nonminimal trefoil births}
\label{subsec:frozen-trefoil}

The preceding Hamiltonian construction can be transferred to a nontrivial
knot type by a \emph{protected-port} connected sum.  The choice of the
nontrivial summand matters: an arbitrary minimal representative may have
neutral BFACF moves, and new shortening moves may be created at the join.
Here we give one explicit join for which both possibilities are excluded.
The assertion concerns the \emph{standard discrete} BFACF system and does
not assert ropelength local minimality after smoothing.

Write $P=P_1$ for the $64$-edge cycle of
\cref{subsec:frozen-unlinks}, and let $Q$ be the $24$-edge polygon
starting at $(0,0,0)$ and encoded by the cyclic direction word
\[
 \texttt{000113342335500022133544}.
\]
The digit convention is that of the preceding subsection.  The word
is a minimal trefoil representative, of the chirality fixed in
\cref{subsec:trefoil-births}.  Its type-$0$ component at the minimal
layer is a singleton: direct coordinate tests also give no shortening
or neutral move for $Q$.

In the original coordinates of $Q$, select the \emph{specific} directed
port edge $[(0,2,1),(1,2,1)]$.  Apply the proper lattice rotation
\[
 T(x,y,z)=(-y,-z,x)
\]
and then the translation $v=(6,1,0)$.  Thus the chosen port becomes
\[
 e_Q=[(4,0,0),(4,0,1)],
\]
and the transformed polygon $T(Q)+v$ lies in $x\ge4$.  This choice is
essential: other exposed vertical edges need not give a frozen splice,
and \cref{rem:port-search} shows that frozen joins are rare.
In particular, at $(4,0,0)$ the two incident edges lead to $(4,0,1)$
and $(4,1,0)$; moreover $(4,1,1)$ is occupied by the transformed
$Q$, blocking the otherwise available corner flip at the interface.
The code in \path{anc/frozen_trefoils/verify_connected_sum_public.py}
reconstructs exactly this placement.  No mirror reflection is used.  Recall that the right port of $P$ is
$e_P=[(3,0,0),(3,0,1)]$.  Delete $e_P$ and $e_Q$, and insert
\[
 b_z=[(3,0,z),(4,0,z)],\qquad z=0,1.
\]
Call the resulting polygon $R_1$.  The plane $x=7/2$ meets it only
in the two bridging edges; capping each side recovers $P$ and $Q$.
Consequently $[R_1]=[P]\#[Q]=0_1\#3_1=3_1$ and
$\ell(R_1)=64+24=88$.

\begin{theorem}[Infinitely many BFACF-frozen trefoils]
\label{thm:infinitely-many-frozen-trefoils}
For every integer $m\ge1$ and either trefoil chirality, there is a
simple-cubic polygon $R_m$ of that chirality such that
\[
\ell(R_m)=64m+24,\qquad
\#\mathcal M_0(R_m)=\#\mathcal M_{-2}(R_m)=0,
\]
where $\mathcal M_j(R)$ denotes the set of legal BFACF moves of
length change $j$ from $R$.  Thus $R_m$ is an isolated vertex of
$G_{64m+24}(3_1)$, and for every minimal trefoil $Q_*$ of its chirality,
\[
 H(R_m,Q_*)\ge 64m+26.
\]
The trefoil of either chirality has infinitely many nonminimal strong
local-minimum birth components, and its capped BFACF graphs are not
eventually connected.
\end{theorem}
\begin{proof}
Take $m$ copies of the Hamiltonian unknot $P$ in the consecutive
four-plane slabs of \cref{thm:infinitely-many-frozen-unknots}, splice
adjacent blocks at their matching vertical ports, and attach the
rotated and translated copy of $Q$ at the last right port, now at
$x=4m-1$.  The final pair of bridges lies in the unit plaquette
between $x=4m-1$ and $x=4m$ at $y=0$ and $0\le z\le1$.
All modifications are made on disjoint ports and preserve the edge
count: the resulting polygon has $64m+24$ edges.  Separating planes
through the interfaces meet the polygon in exactly two points;
successive capping exhibits its knot type as
$0_1^{\#m}\#3_1=3_1$.

For the local-move claim, a neutral BFACF move is a two-edge corner
flip in one unit plaquette, and a shortening move replaces a
three-edge plaquette detour by its fourth edge.  Away from the
interfaces the move neighborhood agrees with that of $P$ or $Q$,
neither of which admits a nonincreasing move.  Any possible new move
must therefore meet one of the splicing plaquettes or an adjacent
unit plaquette.  Such a move is determined by finitely many edges
within one unit of a joining port.  The junctions between two
Hamiltonian blocks have already been checked in
\cref{thm:infinitely-many-frozen-unknots}; the terminal junction
between the Hamiltonian block and $Q$ has the same local coordinates
for every $m$.  These types of junction lie in disjoint unit neighborhoods:
the $x$-coordinates of consecutive splice planes differ by four.
No move involving both junctions is possible.

For completeness, the explicit coordinate verifier
\path{anc/frozen_trefoils/verify_connected_sum_public.py}
regenerates the full polygon from the two cyclic words and checks
\emph{every} two-edge corner flip and three-edge shortening detour.
It returns zero moves of either type for the terminal-join model
$m=1$, for the internal Hamiltonian joins, and for the additional
examples $m=2,3,4,8,16$.  The exhaustive local interface check,
combined with the separation of junction neighborhoods, applies to
arbitrary $m$, not just the sampled values.  Hence every move out of
$R_m$ increases the edge count by two, and its component at its birth
length $64m+24$ is a singleton.  The length bound follows from the
first move of any path to a shorter representative.

Distinct $m$ give distinct unbounded birth lengths.  By
\cref{prop:eventual-connectivity}, the capped trefoil graphs cannot
be eventually connected.  Reflection of the construction gives the
other chirality; reflection preserves the BFACF move types and counts.
\end{proof}

\begin{remark}[Scope of the connected-sum construction]
\label{rem:frozen-trefoil-scope}
The lower bound $64m+26$ need not be the exact merge height for
$m\ge1$; no upper bound of the same value is asserted.
The isolated length-$88$ trefoil is a different phenomenon from the
selected length-$26$ nonminimal trefoil components studied below.
This protected-port construction does not by itself give frozen
configurations of other knot types.  These are obtained by the different
block substitution of \cref{subsec:frozen-all}.  Neither construction
gives a positive barrier for unit-thickness continuous knots.
\end{remark}

\begin{table}[ht]
\centering
\begin{tabular}{rrrr}
\toprule
Blocks $m$ & Edges $64m+24$ & Legal type $0$ & Legal type $-2$\\
\midrule
$1$&$88$&$0$&$0$\\
$2$&$152$&$0$&$0$\\
$3$&$216$&$0$&$0$\\
$4$&$280$&$0$&$0$\\
$8$&$536$&$0$&$0$\\
$16$&$1048$&$0$&$0$\\
\bottomrule
\end{tabular}
\caption{Independent coordinate checks on the protected-port frozen-trefoil
family.  The proof of the universal statement uses locality of the
junction checks; these six finite tests alone do not imply it.}
\label{tab:frozen-trefoil-checks}
\end{table}

\begin{remark}[How rare protected ports are]
\label{rem:port-search}
Keep the right port $e_P$ of $P$ fixed.  A \emph{port-aligned join} is
specified by a proper-rotation class of $24$-edge trefoils of the fixed
chirality, one of the $24$ proper cubic rotations, and an edge of the
rotated polygon that is parallel to the $z$-axis and lies in its face of
least $x$-coordinate.  A translation moves this edge to
$e_Q=[(4,0,0),(4,0,1)]$ with the polygon in $x\ge4$, and the splice above
is performed.  For the $75$ minimal classes there are exactly $2448$
port-aligned joins, each a single $88$-edge polygon.  Their nonincreasing
BFACF moves are distributed as follows: $2138$ admit only neutral moves,
$304$ admit both neutral and shortening moves, $2$ admit only shortening
moves, and exactly $4$ admit neither.  All four frozen joins use the
summand $Q$, whose type-$0$ component is a singleton; they represent two
distinct vertices of the proper-cubic quotient, one of which is $R_1$.
Even with $Q$ as summand, only four of its forty port-aligned joins are
frozen.  These counts concern this explicit search space only; they are
not a census of all connected sums.  The search is reproduced by
\path{anc/independent_recheck/splice_search.py}, which uses a BFACF move
generator written independently of the search code used to discover $R_1$.
\end{remark}

\begin{figure}[ht]
\centering
\includegraphics[width=0.78\linewidth]{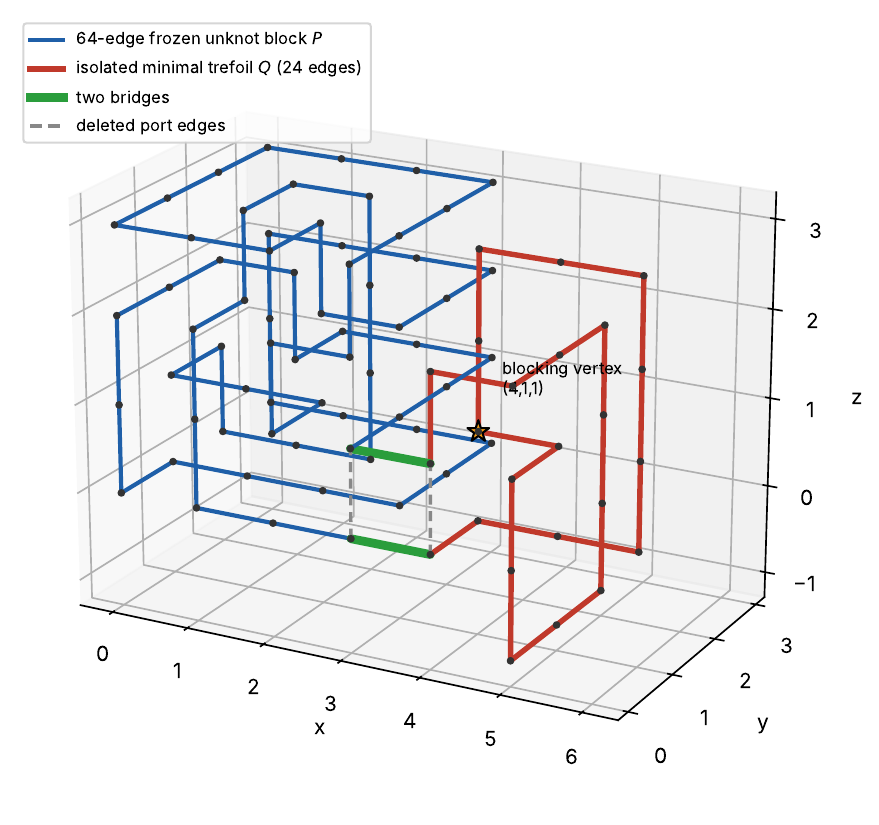}
\caption{The $88$-edge frozen trefoil $R_1$.  The $64$-edge Hamiltonian
block $P$ (blue) and the isolated minimal trefoil $Q$ (red) are joined by
two bridges (green) after deleting the two port edges (dashed).  The
occupied vertex $(4,1,1)$ of $Q$ (star) blocks the corner flip that
would otherwise be available at the interface.  The picture is a
projection of the lattice polygon, not a knot diagram.}
\label{fig:frozen-trefoil88}
\end{figure}

\subsection{Frozen polygons of every knot type: block substitution}
\label{subsec:frozen-all}

The constructions of \cref{subsec:frozen-unlinks,subsec:frozen-trefoil}
splice blocks along one axis and attach one carefully chosen summand.
We now replace \emph{every} vertex of an arbitrary lattice polygon by a
Hamiltonian block.  This settles the existence questions left open above
for all knot types simultaneously.

\paragraph{Why side length five.}
With side length four, no rule that fixes one port vertex per face can
serve all turning patterns.  The splices of
\cref{subsec:frozen-unlinks,subsec:frozen-trefoil} avoid this by using
special ports.  The grid graph of $\{0,1,2,3\}^3$ is bipartite with $32$ vertices
of each colour, so every Hamiltonian path joins vertices of opposite colours.
Fix one port vertex on each of the six faces.  Two of these six ports have
the same colour, so the two faces carrying them cannot be the entry and exit
faces of one Hamiltonian path.  For side length five there is no such
obstruction.  The grid graph of $\{0,\dots,4\}^3$ has $63$ even and $62$
odd vertices.  Every Hamiltonian path therefore begins and ends at even
vertices, and all six face centres are even.  Moreover, if blocks are placed
at $5q+\{0,\dots,4\}^3$, the local parity of a block differs from the global
parity by $q_1+q_2+q_3$.  Consecutive blocks therefore have opposite parity
shifts, which is exactly what is required for an exit port and the next entry
port to be adjacent lattice points.

\paragraph{Block paths.}
Put $B=\{0,1,2,3,4\}^3$ and $c=(2,2,2)$.  For a unit vector
$u\in\{\pm e_1,\pm e_2,\pm e_3\}$, let $p(u)=c+2u$ denote the centre of
the face of $B$ with outer normal $u$.

\begin{definition}[Frozen block path]
\label{def:frozen-block-path}
Let $u\neq v$ be unit coordinate vectors.  A \emph{frozen block path of
type $(u,v)$} is a sequence $\gamma=(\gamma_0,\dots,\gamma_{124})$ of
lattice points satisfying the following conditions.
\begin{enumerate}[label=(B\arabic*)]
\item $\gamma$ is a Hamiltonian path in the grid graph of $B$, with
$\gamma_0=p(u)$ and $\gamma_{124}=p(v)$.
\item $\gamma_1=c+u$ and $\gamma_{123}=c+v$; that is, the first and last
edges are normal to the entry and exit faces.
\item No three consecutive edges of $\gamma$ are three sides of one unit
plaquette.
\item The polygonal arc $|\gamma|$ is unknotted in the ball
$[0,4]^3$: the pair $([0,4]^3,|\gamma|)$ is homeomorphic to a ball with a
straight diameter.
\end{enumerate}
\end{definition}

Let $\gamma^{\mathrm{str}}$ and $\gamma^{\mathrm{cor}}$ start at
$(0,2,2)=p(-e_1)$ and follow the following $124$-letter direction words
(digits as in \cref{subsec:frozen-unlinks}; each word is printed in two
rows, to be concatenated):
\[
\begin{aligned}
\gamma^{\mathrm{str}}:\ &
\texttt{02120454453535111220122344040011135553443424201550400122115153}\\[-2pt]
&\texttt{45404433311113222444240005111355510042434531224040221111535440},\\
\gamma^{\mathrm{cor}}:\ &
\texttt{05454000111132333222044445051550112335434422221115501202444400}\\[-2pt]
&\texttt{11115554423213343455115100442040422312111554024453453321201511}.
\end{aligned}
\]
The first ends at $(4,2,2)=p(e_1)$ and the second at $(2,4,2)=p(e_2)$.

\begin{lemma}[Two base blocks suffice]
\label{lem:base-blocks}
$\gamma^{\mathrm{str}}$ is a frozen block path of type $(-e_1,e_1)$ and
$\gamma^{\mathrm{cor}}$ is a frozen block path of type $(-e_1,e_2)$.
Consequently, for every ordered pair $u\neq v$ of unit coordinate vectors
there is a frozen block path $\gamma_{u,v}$ of type $(u,v)$.
\end{lemma}

\begin{proof}
Conditions (B1)--(B3) are finite coordinate checks on $125$ points and $124$
edges.  For (B4), close $|\gamma^{\mathrm{str}}|$ by the lattice arc
\[
\delta^{\mathrm{str}}:\ (4,2,2)\to(5,2,2)\to(5,2,-1)\to(-1,2,-1)\to(-1,2,2)\to(0,2,2)
\]
in the plane $y=2$, and close $|\gamma^{\mathrm{cor}}|$ by
\[
\delta^{\mathrm{cor}}:\ (2,4,2)\to(2,5,2)\to(-1,5,2)\to(-1,2,2)\to(0,2,2)
\]
in the plane $z=2$; in each case the intermediate segments are unit edges.
Each closing arc $\delta$ meets $[0,4]^3$ only in its endpoints.  Together
with one of the two boundary arcs of the square $\Pi\cap[0,4]^3$, where
$\Pi$ is its plane, it bounds a disk in $\Pi$ disjoint from $(0,4)^3$.
Hence $\delta$ is a trivial arc in the complementary ball
$S^3\setminus(0,4)^3$.  A ball-arc pair glued to a trivial ball-arc pair
along their common boundary has the same knot type as the closed curve.
Therefore (B4) is equivalent to the statement that the $138$- and
$132$-edge closed lattice polygons $\gamma\cup\delta$ are unknots.  Both
facts are certified in \path{anc/frozen_all_knots/} by explicit BFACF paths
to the unit square.  The paths have $1254$ and $1344$ moves and never exceed
$140$ and $134$ edges, respectively.  Every transition is re-verified by a
coordinate checker written independently of the program that found the
paths.  BFACF moves are ambient isotopies, so both closed polygons are
unknots.  As a second check, generic projections in three directions
simplify, by SnapPy's Reidemeister simplification, to diagrams with no
crossings.

For the second assertion let $g$ be a signed permutation matrix with
$g(-e_1)=u$, and with $g(e_1)=v$ if $v=-u$, or $g(e_2)=v$ otherwise.  Put
$\gamma_{u,v}=c+g(\gamma-c)$ with $\gamma=\gamma^{\mathrm{str}}$ or
$\gamma^{\mathrm{cor}}$ respectively.  The map $x\mapsto c+g(x-c)$
preserves $B$, lattice adjacency and unit plaquettes, and maps the face
centre $p(w)$ to $p(gw)$.  It is a homeomorphism of $[0,4]^3$, so it also
preserves (B4).
\end{proof}

All thirty oriented block paths $\gamma_{u,v}$ used below are fixed once
and for all by the rule in the proof.  They are regenerated and rechecked
in the ancillary code.  \Cref{fig:block-paths} shows the two base paths.

\begin{figure}[ht]
\centering
\includegraphics[width=\linewidth]{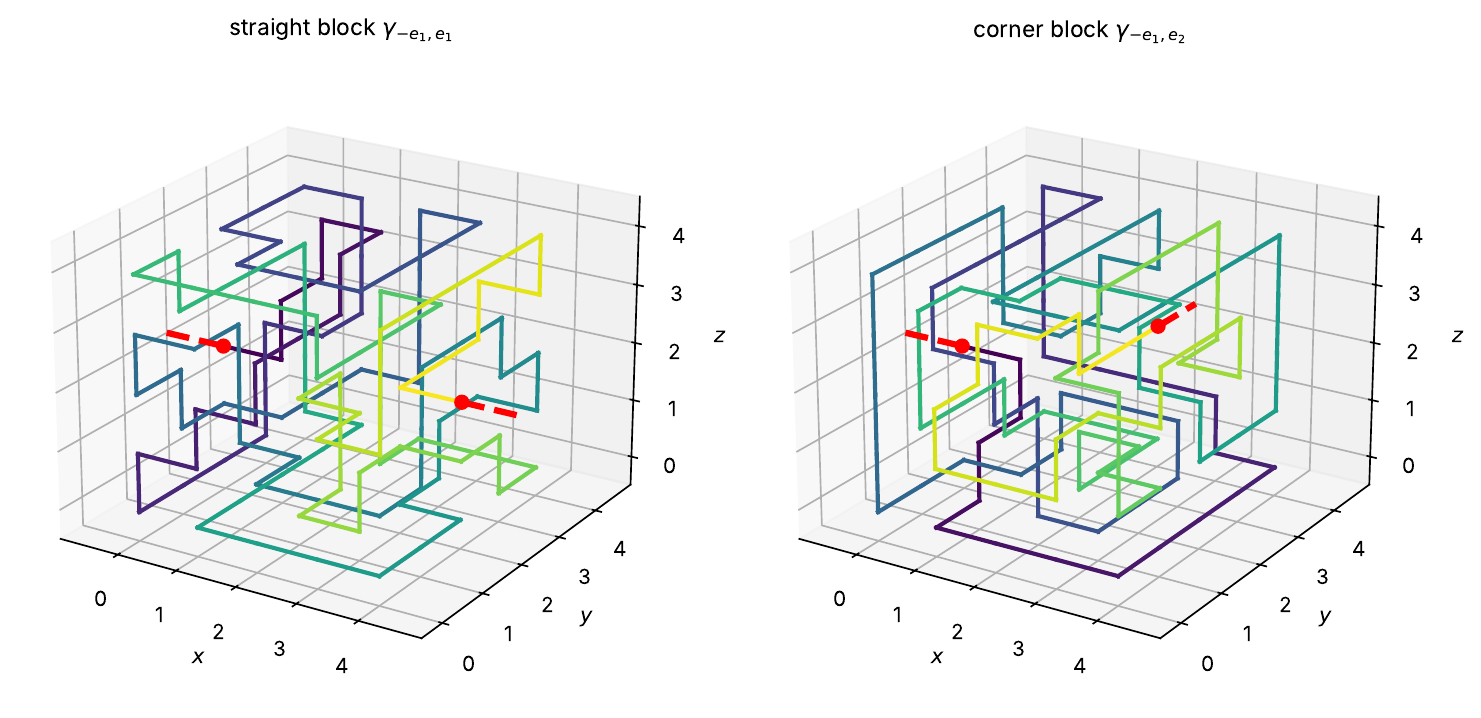}
\caption{The two base Hamiltonian block paths in $\{0,\dots,4\}^3$.  The
colour runs from the entry port $(0,2,2)$ to the exit port.  The dashed red
edges indicate the normal edges joining a block to its neighbours.  The
picture is a projection, not a knot diagram.}
\label{fig:block-paths}
\end{figure}

\begin{definition}[Block substitution]
\label{def:block-substitution}
Let $Q=(q_0,\dots,q_{m-1})$ be an oriented simple-cubic polygon, with
indices modulo $m$, and let $d_i=q_{i+1}-q_i$.  Put
\[
 B_i=5q_i+B,\qquad c_i=5q_i+c,\qquad
 \gamma^{(i)}=5q_i+\gamma_{-d_{i-1},\,d_i}.
\]
The \emph{block substitution} $\Sigma(Q)$ is the closed lattice path
obtained by concatenating the paths $\gamma^{(i)}$ for $i=0,\dots,m-1$ in
cyclic order.  The last point of $\gamma^{(i)}$ is joined to the first
point of $\gamma^{(i+1)}$ by a unit edge.
\end{definition}

The joining edge is well defined: the last point of $\gamma^{(i)}$ is
$c_i+2d_i$, and the first point of $\gamma^{(i+1)}$ is
$c_{i+1}-2d_i=c_i+3d_i$.  The entry and exit faces of a block differ,
since $d_i\neq-d_{i-1}$ for a self-avoiding polygon.

\begin{lemma}[Unknotted arcs with common ends]
\label{lem:trivial-arcs}
Let $\alpha,\beta$ be unknotted properly embedded arcs in a PL $3$-ball
$D$ with $\partial\alpha=\partial\beta$.  Then $\alpha$ and $\beta$ are
ambient isotopic in $D$ by an isotopy fixed on $\partial D$.
\end{lemma}

\begin{proof}
Both pairs are homeomorphic to the standard pair, so there is a
homeomorphism $h\colon D\to D$ with $h(\alpha)=\beta$.  Composing with a
symmetry of the standard pair (a rotation about the diameter, a reflection
containing it, or a half-turn exchanging its ends), we may assume that
$h|_{\partial D}$ preserves orientation and fixes the two endpoints.  The
pure mapping class group of the sphere with two marked points is trivial.
Hence $h|_{\partial D}$ is isotopic to the identity relative to the
endpoints.  After inserting this isotopy in a collar of $\partial D$ in
which $\beta$ is radial, we may assume $h|_{\partial D}=\mathrm{id}$.
The Alexander trick then gives an isotopy rel $\partial D$ from the
identity to $h$, carrying $\alpha$ to $\beta$.
\end{proof}

\begin{theorem}[Frozen polygons of every knot type]
\label{thm:frozen-every-knot}
Let $Q$ be an oriented simple-cubic polygon with $m$ edges, of oriented
knot type $K$.  Then $\Sigma(Q)$ is a self-avoiding simple-cubic polygon such
that:
\begin{enumerate}[label=(\alph*)]
\item $\ell(\Sigma(Q))=125m$, and its vertex set is
$R(Q)=\bigcup_iB_i$;
\item $\Sigma(Q)$ is ambient isotopic to the dilated polygon $5Q+c$ by an
orientation-preserving isotopy; in particular, it has the oriented knot type
$K$, including chirality;
\item $\Sigma(Q)$ admits no legal BFACF move of length change $0$ or $-2$;
\item if $\Sigma(Q)$ and $\Sigma(Q')$ represent the same vertex of
$\mathcal P^{\SC}(K)$, then the underlying unoriented polygons $Q$ and $Q'$
represent the same vertex as well.
\end{enumerate}
\end{theorem}

\begin{proof}
(a) The blocks $B_i$ are pairwise disjoint, because the $q_i$ are distinct
and $5q_i+B$ are disjoint translates of a cube of side length four.  Each
$\gamma^{(i)}$ visits each point of $B_i$ exactly once, and the $m$ joining
edges are unit edges between consecutive blocks.  Thus $\Sigma(Q)$ is a
simple closed lattice polygon with $125m$ vertices, and its vertex set is
$R(Q)$.

(b) The polygon meets $B_i$ exactly in $|\gamma^{(i)}|$: the open joining
edges lie in the unit gaps between consecutive blocks.  By (B4) and
\cref{lem:trivial-arcs}, there is an isotopy supported in $B_i$ and fixed on
$\partial B_i$, replacing $|\gamma^{(i)}|$ by the two segments
$[c_i-2d_{i-1},c_i]\cup[c_i,c_i+2d_i]$.  These isotopies have disjoint
supports.  After performing all of them, the polygon is the union of the
segments
\[
 [c_i,c_i+2d_i]\cup[c_i+2d_i,c_i+3d_i]\cup[c_i+3d_i,c_{i+1}]
 =[c_i,c_{i+1}],
\]
that is, the dilated and translated polygon $5Q+c$, traversed in the same
direction.  A positive dilation followed by a translation does not change the
oriented knot type.  No reflection is used.

(c) A neutral move is a corner flip at a vertex $V$ with neighbours $A,B$
along perpendicular edges, and it requires the fourth corner $W=A+B-V$ of the
unit plaquette to be unoccupied.  A shortening move requires three consecutive
edges forming three sides of a unit plaquette, so each two consecutive edges
among them are perpendicular.  By (B2), the edges immediately before and after
each joining edge $[c_i+2d_i,c_i+3d_i]$ are parallel to it.  Hence no corner
has a joining edge, and no shortening triple contains a joining edge.  Every
corner $A,V,B$ therefore lies in a single block $B_i$.  The fourth corner $W$
lies in the bounding box of $\{A,V,B\}$, hence in $B_i\subset R(Q)$, so it is
occupied by (a).  Every shortening triple lies in a single $\gamma^{(i)}$
and is excluded by (B3).

(d) Let $\phi$ be a proper lattice isometry with
$\phi(\Sigma(Q))=\Sigma(Q')$.  Then $\phi(R(Q))=R(Q')$.  The block
decomposition of $R(Q)$ is determined by the set $R(Q)$ alone.  If
$a_j=\min\{x_j : x\in R(Q)\}$, then the blocks are separated by the planes
$x_j=a_j-\tfrac12+5k$, $k\in\mathbb Z$, and
$\max\{x_j : x\in R(Q)\}\equiv a_j+4\pmod 5$.  The linear part of $\phi$
is a signed permutation, so $\phi$ maps the minimum or maximum of one
coordinate on $R(Q)$ to the minimum or maximum of a coordinate on $R(Q')$.
Hence $\phi$ maps the family of separating planes of $R(Q)$ onto that of
$R(Q')$, and each block $B_i$ onto some block $B'_j$.  It therefore
induces a proper lattice isometry $\hat\phi$ with
$\hat\phi(\{q_i\})=\{q'_j\}$.  The cyclic order in which $\Sigma(Q)$ visits
its blocks is the cyclic order of $Q$.  Since $\phi$ preserves the polygon,
$\hat\phi$ maps $Q$ onto $Q'$ up to orientation.
\end{proof}

\begin{corollary}[Births, Gordian configurations and eventual connectivity for every knot type]
\label{cor:every-knot}
Let $K$ be any knot type, of either chirality, and write $n_K=n_{\SC}(K)$.
\begin{enumerate}[label=(\roman*)]
\item For every even $m\geq n_K$ there is a BFACF-frozen polygon of type
$K$ and length $125m$.  It is an isolated vertex of $G_{125m}(K)$, and its
singleton is a strong local-minimum component born at $125m>n_K$.
\item The number of strong local-minimum components of $K$ born at length
$125m$ is at least the number of vertices of $\mathcal P^{\SC}(K)$ of length
exactly $m$.  In particular, $K$ has infinitely many nonminimal strong
local-minimum components, and their birth lengths are unbounded.
\item Every $P$ in~(i) is lattice Gordian, with
\[
 h_K(P)\geq\ell(P)+2,\qquad\text{equivalently}\qquad e_K(P)\geq2.
\]
\item $G_{125m}(K)$ is disconnected for every even $m\geq n_K$.  Hence there
is no $N_0(K)$ such that $G_N(K)$ is connected for all $N\geq N_0(K)$.
\end{enumerate}
\end{corollary}

\begin{proof}
Every simple-cubic polygon admits a $+2$ move.  Let $x_{\max}$ be its largest
first coordinate and let $V$ be a vertex attaining it.  At most one neighbour
of $V$ equals $V-e_1$, so some edge at $V$ lies in the plane $x=x_{\max}$.
Pushing this edge in the direction $+e_1$ uses two unoccupied vertices.
Hence $\mathcal P^{\SC}(K)$ contains polygons of every even length
$m\geq n_K$.

(i) Apply \cref{thm:frozen-every-knot} to such a $Q$.  Since $\Sigma(Q)$ has
no nonincreasing move, it has no neighbour in $G_{125m}(K)$.  We have
$125m>m\geq n_K$.  Its singleton is a type-$0$ component with no shortening
move, that is, a strong local-minimum component, born at $125m$ by
\cref{prop:birth-criterion}.

(ii) Choose an orientation of each vertex $[Q]$ of length $m$.  By
\cref{thm:frozen-every-knot}(d), distinct classes $[Q]$ give distinct
isolated vertices $[\Sigma(Q)]$.  Unboundedness follows from (i).

(iii) Every BFACF path starting at $P$ begins with a $+2$ move, and
$\mathcal M_K$ consists of shorter polygons.  This is the inequality, and
also \cref{prop:gordian-birth}.

(iv) $G_{125m}(K)$ contains the isolated vertex of~(i) and the nonempty
minimal layer, which is a different vertex set.  The final assertion is
(i) combined with \cref{prop:eventual-connectivity}.
\end{proof}

\begin{remark}[Relation to the earlier families and scope]
\label{rem:frozen-all-scope}
For $K=0_1$ and $3_1$, \cref{cor:every-knot} recovers the qualitative
conclusions of
\cref{thm:infinitely-many-frozen-unknots,thm:infinitely-many-frozen-trefoils}
with different birth lengths.  The earlier explicit families remain useful,
since they have shorter lengths and an exactly computed base barrier
$H(P_1,\square)=66$.  The proof of \cref{thm:frozen-every-knot} is
symbolic: the only finite computation is the verification of the two
$125$-point base paths in \cref{lem:base-blocks}.  The theorem gives only
the uniform lower bound $e_K(\Sigma(Q))\geq2$.  It says nothing about
\cref{prob:unbounded-excess}, about exact merge levels of $\Sigma(Q)$, or
about thick-rope Gordian knots in the continuous setting.  Finally, the
substitution and its proof use only proper lattice isometries of the
blocks.  In particular, both chiralities of every chiral knot type are
treated directly, without reflecting the construction.
\end{remark}

\begin{remark}[Iterated substitution]
\label{rem:iterated-substitution}
Since $\Sigma(Q)$ is again an oriented simple-cubic polygon, the construction
can be iterated: put $\Sigma^0(Q)=Q$ and $\Sigma^k(Q)=\Sigma(\Sigma^{k-1}(Q))$.
By \cref{thm:frozen-every-knot}, $\Sigma^k(Q)$ is frozen for $k\ge1$, has the
oriented knot type of $Q$, and has length $125^k\,\ell(Q)$.  Its vertex set is
described exactly by base-$5$ digits:
\[
 V\bigl(\Sigma^k(Q)\bigr)=\bigcup_{q\in V(Q)}\bigl(5^kq+\{0,1,\dots,5^k-1\}^3\bigr).
\]
This holds for $k=0$, and $5\cdot\{0,\dots,5^{k-1}-1\}+\{0,\dots,4\}=\{0,\dots,5^k-1\}$
gives the induction step via \cref{thm:frozen-every-knot}(a).  Consequently the
rescaled polygons $5^{-k}\Sigma^k(Q)$ have vertex sets
$\bigcup_q\bigl(q+5^{-k}\{0,\dots,5^k-1\}^3\bigr)$.  They converge in the
Hausdorff metric to the union of the closed unit cubes $q+[0,1]^3$,
$q\in V(Q)$.  Thus the iterates form a Peano-type sequence of knotted
lattice polygons that fill the cubical neighbourhood of $Q$ while keeping
its knot type.  For the $24$-edge trefoil, $\Sigma^2(Q)$ has $375{,}000$
edges.  A direct coordinate check confirms the vertex-set formula for
$k=1,2$ and finds no move of type $0$ or $-2$.

The iterates give no conclusion beyond \cref{cor:every-knot}, because
$\Sigma^k(Q)=\Sigma(Q')$ with $Q'=\Sigma^{k-1}(Q)$.  In particular, the
hierarchy of scales does not by itself yield a larger escape excess: every
$\Sigma^k(Q)$ is built from the same $5\times5\times5$ blocks at its finest
scale, and freezing controls only the first move
(\cref{prob:excess-substitution}).
\end{remark}

\begin{figure}[ht]
\centering
{\small (a) $Q$: minimal trefoil, $24$ edges\hspace{0.2\linewidth}(b) $\Sigma(Q)$: frozen trefoil, $125\cdot24=3000$ edges}\\[3pt]
\includegraphics[width=0.95\linewidth]{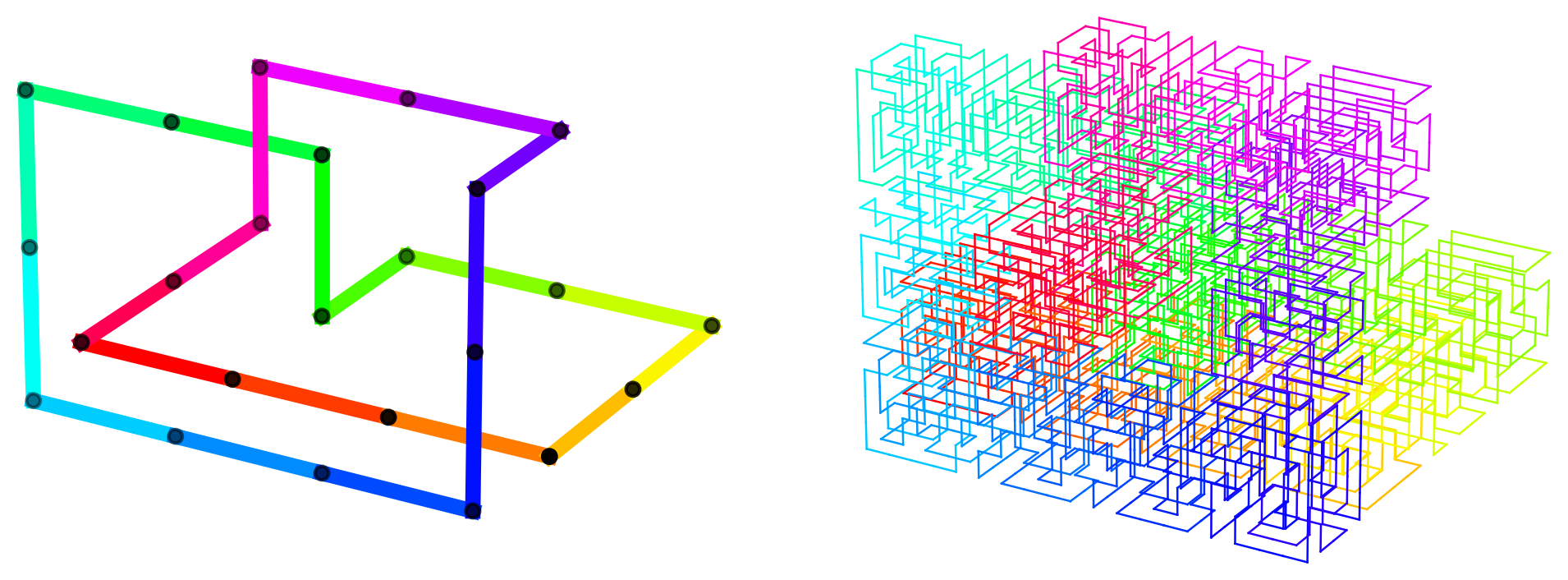}
\caption{Block substitution of the minimal $24$-edge trefoil $Q$ of
\cref{subsec:frozen-trefoil}.  Both pictures are coloured by arc length along
the knot, so equal colours in (a) and (b) correspond.  The polygon in (b)
fills the $24$ blocks $5q+\{0,\dots,4\}^3$ and admits no move of type $0$ or
$-2$.  The pictures are projections, not knot diagrams; the knot type of (b)
is given by \cref{thm:frozen-every-knot}(b).}
\label{fig:frozen-trefoil3000}
\end{figure}

\begin{table}[ht]
\centering
\small
\begin{tabular}{lrrrrl}
\toprule
Type of $Q$ & $\ell(Q)$ & $\ell(\Sigma(Q))$ & Type $0$ & Type $-2$ & SnapPy comparison\\
\midrule
$3_1$, minimal ($Q$ above) & 24 & 3000 & 0 & 0 & $3$ crossings, writhe $+3$\\
$4_1$, minimal & 30 & 3750 & 0 & 0 & isometric\\
$3_1$ & 36 & 4500 & 0 & 0 & $3$ crossings\\
$4_1$ & 32 & 4000 & 0 & 0 & ---\\
$5_1$ & 46 & 5750 & 0 & 0 & $5$ crossings, writhe $+5$\\
$5_2$ & 40 & 5000 & 0 & 0 & isometric\\
$6_1$ & 50 & 6250 & 0 & 0 & isometric\\
$6_2$ & 46 & 5750 & 0 & 0 & ---\\
$6_3$ & 52 & 6500 & 0 & 0 & ---\\
$7_1$ & 70 & 8750 & 0 & 0 & ---\\
$8_{19}$ & 50 & 6250 & 0 & 0 & ---\\
$3_1\#3_1$ & 50 & 6250 & 0 & 0 & ---\\
$3_1\#3_1^{*}$ & 52 & 6500 & 0 & 0 & ---\\
\bottomrule
\end{tabular}
\caption{Direct coordinate checks of block substitutions.  The first two
polygons $Q$ are taken from this paper.  The others were built from closed
braids and shortened by nonincreasing BFACF moves.  The knot type of each
$Q$ was identified by SnapPy~\cite{SnapPy}.  Every row has one closed cycle
occupying all $125\,\ell(Q)$ points of $R(Q)$, and the same holds for
$\Sigma$ of the reversed polygon.  In the last column, ``isometric'' means
that SnapPy finds the complements of $Q$ and $\Sigma(Q)$ isometric.  For
torus knots, the column records the simplified diagrams, which agree for
$Q$ and $\Sigma(Q)$.  A dash indicates that only the coordinate checks were run.
These checks illustrate \cref{thm:frozen-every-knot}; its proof does not
depend on them.}
\label{tab:frozen-all-checks}
\end{table}

\subsection{The trefoil: minimal and nonminimal branches}
\label{subsec:trefoil-births}

Fix the trefoil chirality represented by the archived $24$-edge seed.
Throughout this subsection write $G_N=G_N^{\SC,\BFACF}(3_1)$ and take
the quotient by translations and the $24$ proper cubic rotations.
Cyclic starting points and traversal reversal represent the same polygon;
spatial reflections are not identified.  All edge counts below are
undirected and exclude loops.  We abbreviate the minimax height by
\[
 H(P,Q)=\min_{\Gamma:P\leadsto Q}\max_{R\in\Gamma}\ell(R).
\]
For components already connected at their respective birth levels, the
first cap joining them is denoted by $H(D,A)$.

It is known that $n_{\SC}(3_1)=24$
\cite{DiaoMinimal,DiaoSmallest,MinimalKnotsRechnitzer}.
The known minimal population is $1664$ translation classes of one
chirality.  The verified minimal list realizes this population in $75$
proper-rotation classes.  Its three type-$0$ components, denoted
$C_1,C_2,C_3$, have sizes $56,18,1$ and respectively $127,27,0$ edges.
They all lie in a single component $A$ of $G_{26}$, with $5760$ vertices
and $19941$ edges.  Hence their reported-layer merge tree is
$3\to1$ at level $26$.

For comparison, Table~\ref{tab:trefoil-enumeration} records published
translation populations.  At the higher lengths these are reported
sampling results, not exhaustive upper bounds supplied by the present
computation.  The literature's symmetry-class convention is not used
to identify its class counts with our proper-rotation quotient.

\begin{table}[ht]
\centering
\begin{tabular}{cr}
\toprule
Length & Reported translation classes, both chiralities\\
\midrule
24 & 3328\\
26 & 281208\\
28 & 14398776\\
\bottomrule
\end{tabular}
\caption{Trefoil populations from \cite[Table~4]{MinimalKnotsRechnitzer}.
The length-$26$ count is used only as a cross-check and an explicit
conditional completeness hypothesis below.}
\label{tab:trefoil-enumeration}
\end{table}

The earlier prototype search in Table~\ref{tab:seed-basic} remains a
reproducibility test: it finds the $56$-state minimal component, then
stops at $1000$ states under cap $26$.  The component and path certificates
below supersede that capped search for the trefoil merge claims.

\subsubsection{Six certified births at length 26}

The six sets $D_i$ in Table~\ref{tab:trefoil26-components} are specified
by the ancillary files \path{component26_1.txt} through
\path{component26_6.txt}.  Each file lists canonical direction words;
the file names and conventions are described in
\path{anc/trefoil26/README_EN.md}.

\begin{table}[ht]
\centering
\begin{tabular}{crrrr}
\toprule
Component & States & Type-$0$ edges & Translation classes & $H(D_i,A)$\\
\midrule
$D_1$ & 87 & 205 & 2016 & 28\\
$D_2$ & 18 & 27 & 432 & 28\\
$D_3$ & 30 & 59 & 672 & 28\\
$D_4$ & 20 & 28 & 384 & 30\\
$D_5$ & 36 & 72 & 864 & 28\\
$D_6$ & 6 & 6 & 108 & 28\\
\bottomrule
\end{tabular}
\caption{Certified strong local-minimum components of one trefoil
chirality at length $26$.  Translation populations sum the proper-rotation
orbit sizes, including stabilizers.  The table does not presuppose
exhaustiveness among all length-$26$ trefoils.}
\label{tab:trefoil26-components}
\end{table}

\begin{theorem}[Certified nonminimal trefoil merge heights]
\label{thm:trefoil26-heights}
Each $D_i$ is a complete strong local-minimum component of $G_{26}$.
For every $P\in D_i$ and every minimal $24$-edge polygon $Q$ of the fixed
chirality,
\[
 H(P,Q)=
 \begin{cases}
  30,&i=4,\\
  28,&i\in\{1,2,3,5,6\}.
 \end{cases}
\]
In particular, $D_4$ requires excess length $4$ above its own birth level
to reach the minimal layer.
\end{theorem}

\begin{proof}[Computer-assisted proof]
The coordinate-based Python verifier checks every vertex of each $D_i$:
the direction word is canonical and self-avoiding, all legal neutral
neighbors lie in the same list, and no shortening move exists.  It also
checks connectedness and the edge counts in
Table~\ref{tab:trefoil26-components}.  By
\cref{prop:birth-criterion}, these are whole components of $G_{26}$.
They are pairwise disjoint and disjoint from $A$, so their merge heights
with $A$ are at least $28$.

The six archived certificates connect a common minimal root to one
representative in each $D_i$.  They have respectively
$22,22,25,46,15,17$ moves and peak lengths $28,28,28,30,28,28$.
Each transition is checked by a coordinate-based generator independent
of the direction-word replacement generator used in the search.
Connectivity in $D_i$ and in $A$ extends these upper bounds to every
$P\in D_i$ and every minimal $Q$.

For the remaining lower bound, let $B$ and $E$ be the two exhausted
cap-$28$ searches from $A$ and $D_4$, respectively.  Their vertex counts
are given in Table~\ref{tab:trefoil-cap-components}.  The lists are
disjoint.  For every vertex in each list, an independent coordinate-based
C++ implementation verifies all legal neighbors under cap $28$ and
checks that each lies in the same list.  It also compares the complete
neighbor sets with those of the search generator, checks canonicalization
by explicit proper rotations, validates the recorded parent edges,
and verifies connectedness.  Thus $B$ and $E$ are distinct whole
components of $G_{28}$.  Consequently $H(D_4,A)>28$; parity gives
$H(D_4,A)\geq30$, and the explicit path gives equality.  The analogous
cap-$26$ audit verifies $A$, including its published shortest-layer collection.
\end{proof}

\begin{table}[ht]
\centering
\begin{tabular}{lrrrrr}
\toprule
Component & Length 24 & Length 26 & Length 28 & Total & Edges\\
\midrule
$A\subset G_{26}$ & 75 & 5685 & -- & 5760 & 19941\\
$B\subset G_{28}$ & 75 & 5862 & 297897 & 303834 & 1366087\\
$E\subset G_{28}$ & 0 & 20 & 1673 & 1693 & 5269\\
\bottomrule
\end{tabular}
\caption{Complete connected components certified by exhaustive
coordinate-neighborhood audits.  This does not assert that $B\cup E=G_{28}$.}
\label{tab:trefoil-cap-components}
\end{table}

\FloatBarrier
\begin{remark}[Coordinate spans of the two contrasting components]
\label{rem:trefoil-spans}
For a lattice polygon $P$, let $s(P)$ be the nondecreasing rearrangement of
the three coordinate spans
\[
 \max_{v\in V(P)}v_j-\min_{v\in V(P)}v_j,\qquad j=1,2,3.
\]
This triple is invariant under translations and proper cubic rotations,
so it is well-defined on the quotient vertices used here.  Direct checks
of all archived vertices give
\[
 s(P)=(1,3,3)\quad(P\in D_6),\qquad
 s(P)=(3,3,3)\quad(P\in D_4).
\]
Thus all six classes of $D_6$ admit representatives in the width-one slab
$\mathbb Z^2\times\{0,1\}$, whereas none of the twenty classes of $D_4$
fits in a coordinate-aligned width-one slab.  Their merge heights with
the minimal layer are respectively $28$ and $30$, or excess barriers
$2$ and $4$ above their common birth level $26$.
The script \path{anc/trefoil26/verify_geometry.py} checks every span.
This is a geometric comparison of the two certified components; it does
not assert that coordinate width determines the merge barrier.
The prior theorem that the width-one-slab trefoil minimum is $26$
is due to Ishihara et al.\ \cite{IshiharaSlabs}; here the additional data
are the certified neutral component and its merge height in the unrestricted
three-dimensional filtration.
\end{remark}

\FloatBarrier
\subsubsection{The selected merge tree and its persistence}

\begin{corollary}[Nine certified branches]
\label{cor:trefoil-selected-tree}
Track the three births $C_1,C_2,C_3$ at $24$ and the six births $D_i$
at $26$ in the full BFACF filtration.  Their merge tree is
\[
 \bigl(\bigl((C_1,C_2,C_3)_{26},D_1,D_2,D_3,D_5,D_6\bigr)_{28},
 D_4\bigr)_{30}.
\]
The subscripts indicate merge heights.  For a field $\mathbb F$, let
$V_N\subset H_0(G_N;\mathbb F)$ be spanned by the component classes
containing one of these selected births that has appeared by level $N$.
Then $V_\bullet$ is a persistence submodule with interval decomposition
\[
 \mathbb F_{[24,\infty)}\ \oplus\
 \mathbb F_{[24,26)}^{\oplus2}\ \oplus\
 \mathbb F_{[26,28)}^{\oplus5}\ \oplus\
 \mathbb F_{[26,30)}.
\]
Here $\mathbb F_{[a,b)}$ denotes the interval module on integer levels
$a\leq N<b$.  This is a statement about $V_\bullet$, not a formula
for the entire barcode of $H_0(G_\bullet;\mathbb F)$.
\end{corollary}
\begin{proof}
The three minimal components merge in $A$ at $26$.
At cap $26$, the six $D_i$ are distinct from one another and from $A$.
At cap $28$, precisely $D_1,D_2,D_3,D_5,D_6$ among these six lie with
$A$ in $B$, whereas $D_4$ lies in $E$.  The latter two branches merge
at $30$ by \cref{thm:trefoil26-heights}.  Thus the selected component
counts at levels $24,26,28,30$ are respectively $3,7,2,1$.
Inclusions send a selected component class to another selected class,
so $V_\bullet$ is a submodule.  Taking differences of component basis
vectors at each merger gives the displayed interval decomposition.
\end{proof}

\begin{figure}[ht]
\centering
\begin{tikzpicture}[x=1.0cm,y=1.0cm,line width=0.7pt]
 \foreach \y/\n in {0/24,1.25/26,2.5/28,3.75/30}{
   \draw[gray!30,dashed] (-0.5,\y)--(10.0,\y);
   \node[left] at (-0.6,\y) {$\n$};
 }
 \node[above] at (-0.6,4.1) {$N$};
 \draw (0,0)--(0,1.25)--(2,1.25)--(2,0);
 \draw (1,0)--(1,2.5);
 \draw (1,2.5)--(7,2.5);
 \foreach \x/\i in {3/1,4/2,5/3,6/5,7/6}{
   \draw (\x,1.25)--(\x,2.5);
   \fill (\x,1.25) circle (1.7pt);
   \node[below] at (\x,1.25) {$D_{\i}$};
 }
 \draw (4,2.5)--(4,3.75)--(9,3.75)--(9,1.25);
 \draw[->] (6.5,3.75)--(6.5,4.45);
 \foreach \x/\i in {0/1,1/2,2/3}{
   \fill (\x,0) circle (1.7pt);
   \node[below] at (\x,0) {$C_{\i}$};
 }
 \fill (9,1.25) circle (1.7pt);
 \node[below] at (9,1.25) {$D_4$};
 \fill (1,1.25) circle (2pt);
 \fill (4,2.5) circle (2pt);
 \fill (6.5,3.75) circle (2pt);
 \node[left] at (1,1.8) {$A$};
 \node[left] at (4,3.1) {$B$};
 \node[right] at (9,2.5) {$E$};
\end{tikzpicture}
\caption{The nine certified trefoil branches.  Births occur at the
bottom endpoints; horizontal segments join branches at a common merge
height.  The counts $3,7,2,1$ refer to these selected branches only.
Unlisted births, in particular those at levels $28$ and above, are not
excluded by this diagram.}
\label{fig:trefoil-selected-tree}
\end{figure}

\FloatBarrier
\paragraph{A geometric path certificate for the height-30 branch.}
Figure~\ref{fig:trefoil-merge-certificate} displays the reverse of
\path{certificate_C4.txt}.  It starts at $p_0\in D_4$, first attains
length $30$ at $p_7$, and ends at a $24$-edge representative $p_{46}$.
The $46$-move path has been lifted from quotient words to actual lattice
coordinates, so consecutive states differ by a single plaquette move
without intervening rotations.  Display centering uses translations only.
An additional verifier checks each consecutive pair by its edge-set
difference, which must be the boundary of a unit plaquette with a
$1\leftrightarrow3$ or $2\leftrightarrow2$ replacement.  The complete
$47$-state atlas and plotting coordinates are included in the supplement.
The peak height is optimal by \cref{thm:trefoil26-heights}; the number
of moves is not asserted to be shortest.

\begin{figure}[ht]
\centering
\includegraphics[width=\textwidth]{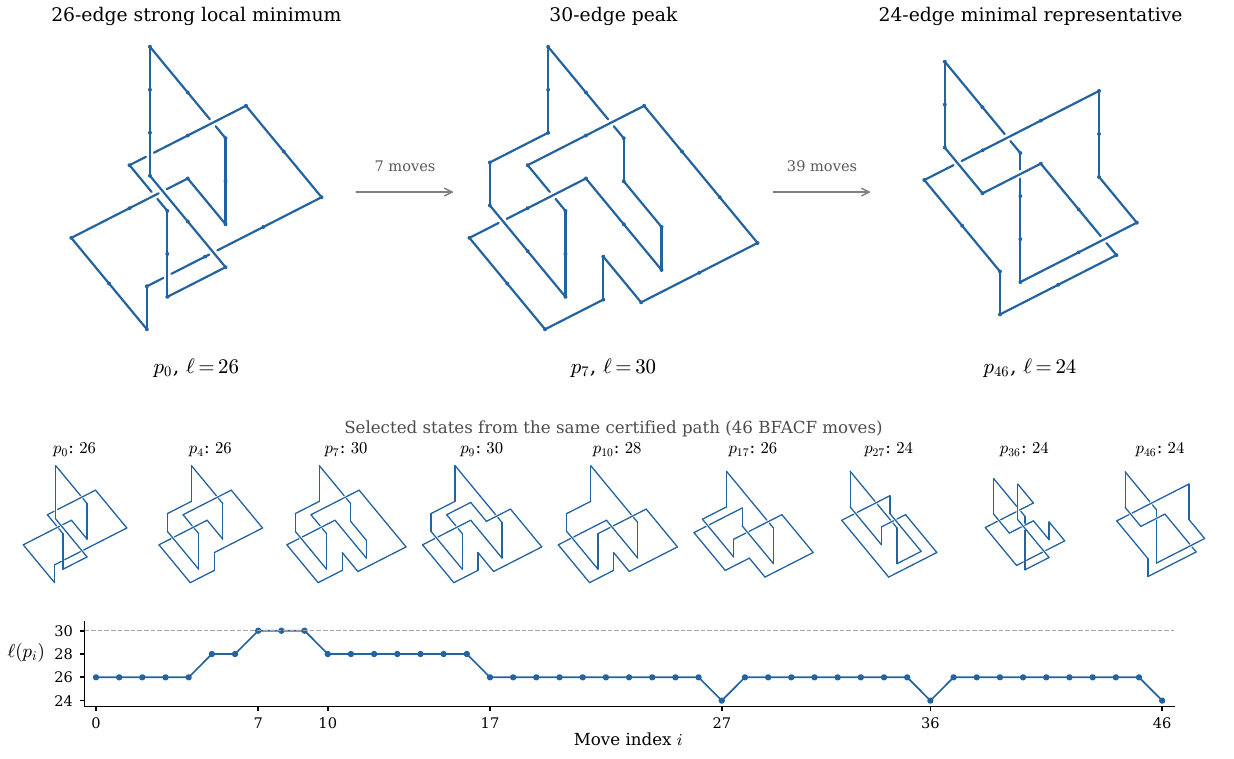}
\caption{A certified trefoil BFACF path from the $26$-edge component
$D_4$ to the minimal layer.  The top row shows the starting polygon,
the first $30$-edge peak, and the $24$-edge endpoint.  The middle row
samples the same $46$-move path; the bottom row records every length.
All polygons use one fixed oblique projection, with gaps on underpassing
strands determined by their three-dimensional depths.  Within each
row a common scale is used.  The diagram supplies the upper bound
$H(D_4,A)\leq30$; the disjoint closed cap-$28$ components establish equality.}
\label{fig:trefoil-merge-certificate}
\end{figure}
\FloatBarrier

\subsubsection{Completeness and the population cross-check}

The verified union $A\cup D_1\cup\cdots\cup D_6$ contains $5882$
proper-rotation classes at length $26$.  Summing their proper-rotation
orbit sizes gives
\[
 136128+(2016+432+672+384+864+108)=140604
\]
translation classes of the fixed chirality.  Doubling gives $281208$,
the value in \cite[Table~4]{MinimalKnotsRechnitzer}.  However, the
methodology in that source uses GAS sampling and accumulation of
configurations.  Agreement with its table is a population cross-check;
it is not an independent certified upper bound on all length-$26$ trefoils.

\begin{proposition}[Conditional completeness through length 26]
\label{prop:trefoil26-completeness}
If the total number of translation classes of length-$26$ trefoils of
the fixed chirality is $140604$, then $G_{26}$ has exactly seven
components, namely $A,D_1,\ldots,D_6$.  Under this hypothesis,
\cref{cor:trefoil-selected-tree} describes all branches born at lengths
at most $26$, followed through their mergers.
\end{proposition}
\begin{proof}
The verified distinct classes already realize the assumed total at
length $26$, and the full length-$24$ layer lies in $A$.
There are no shorter trefoils.  Thus no additional vertex or component
of $G_{26}$ remains.
\end{proof}

Without this population hypothesis, \cref{thm:trefoil26-heights,cor:trefoil-selected-tree}
remain valid, and the full graphs satisfy
\[
 |\pi_0(G_{26})|\geq7,\qquad |\pi_0(G_{28})|\geq2.
\]
Neither $B\cup E=G_{28}$ nor connectivity of $G_{30}$ is asserted.
A certified independent length-$26$ enumeration would remove the
hypothesis in \cref{prop:trefoil26-completeness}; it would still not
classify births at $28$ or above.

\FloatBarrier

\subsection{The figure-eight knot in the simple cubic lattice}

For the figure-eight knot \(4_1\), the minimal simple cubic lattice length is
\[
 n_{\SC}(4_1)=30.
\]
Published enumeration data give
\[
 P_{30}(4_1)=3648,\qquad S_{30}(4_1)=152
\]
\cite{Scharein2009,MinimalKnotsRechnitzer,HongNoOh}.  The published list is compatible with the mirror-sensitive quotient used here:
translation, cyclic reparametrization, traversal reversal, and proper cubic
rotations are identified, while reflections are retained.  Scharein et al. also
enumerated minimal lattice knots and partitioned them into classes under
length-preserving (type-0) BFACF moves \cite{Scharein2009}.  Hence the existence
of nontrivial structure inside the minimal layer is part of the established
lattice-knot literature.  Our question is the next minimax one: at what higher
length cap do the complete birth-level components merge under the full BFACF move
set?

Using the published $152$ minimal classes as a complete birth-level list, our
type-$0$ graph has exactly four components
\[
 C_1,C_2,C_3,C_4,
 \qquad
 (|C_1|,|C_2|,|C_3|,|C_4|)=(58,58,18,18).
\]
The mirror involution exchanges $C_1\leftrightarrow C_2$ and
$C_3\leftrightarrow C_4$.  The two $58$-vertex components have $118$ type-$0$
edges each and the two $18$-vertex components have $26$ type-$0$ edges each.
An exhaustive BFACF search under the cap $N=32$, started independently from a
$58$-vertex component and from an $18$-vertex component, gives the same bounded
state set of $19618$ canonical polygons, consisting of all $152$ length-$30$
minimal classes and $19466$ length-$32$ states.  Therefore all four birth-level
components lie in a single level-$32$ component.

\begin{table}[ht]
\centering
\begin{tabular}{ccccccc}
\toprule
\(K\) & \(n_{\SC}(K)\) & level \(N\) & mode & cap & \(|V|\) & \(|E|\)\\
\midrule
\(3_1\) & 24 & 24 & seed-generated & none & 56 & 127\\
\(3_1\) & 24 & 26 & capped & 1000 & 1000 & 2164\\
\(4_1\) & 30 & 30 & seed-generated & none & 58 & 118\\
\(4_1\) & 30 & 32 & capped & 1000 & 1000 & 2137\\
\(6_3\) & 40 & 40 & seed-generated & none & 7 & 6\\
\(6_3\) & 40 & 42 & seed-generated & none & 12337 & 34146\\
\(5_1\) & 34 & 34 & seed-generated & none & 5 & 4\\
\(5_1\) & 34 & 36 & capped & 1000 & 1000 & 2207\\
\bottomrule
\end{tabular}
\caption{Seed-generated and capped BFACF computations used in the paper.  The
capped rows are partial explorations.  The $6_3$ row at $N=42$ is the complete
component generated from one supplied seed; its reflected component is
isomorphic and disjoint at this level.}
\label{tab:seed-basic}
\end{table}

The comparison between the \(3_1\) and \(4_1\) birth-level seeds suggests that
local BFACF structure and global minimal-layer population are different kinds of
data: the known birth layers have very different enumeration sizes, but the
chosen seed components are of comparable size.

\subsection{The figure-eight knot and its mirror}

The figure-eight knot is amphichiral: topologically, \(4_1\) is equivalent to
its mirror image \(4_1!\).  This makes it a useful test case for separating
ordinary topological equivalence from connectivity inside a length-filtered
lattice move graph.  The distinction is move-system-dependent: the calculation
below concerns BFACF paths, not arbitrary lattice isotopies.

Starting from the 30-edge simple cubic lattice representative used above, we
formed a reflected mirror seed by applying the lattice reflection
\[
 (x,y,z)\longmapsto (-x,y,z).
\]
We then explored the BFACF-reachable graph generated by the two seeds under the
minimal length bound \(\ell\leq 30\).  When mirror symmetries are not identified,
the computation gives \(116\) vertices, \(236\) edges, and two components of sizes
\(58\) and \(58\).  If mirror symmetries are identified during canonicalization,
the same input gives \(58\) vertices, \(118\) edges, and one component.  These two
checks are summarized in Table~\ref{tab:mirror}.

\begin{table}[ht]
\centering
\begin{tabular}{cccccc}
\toprule
input & mirror identified? & \(N\) & \(|V|\) & \(|E|\) & \(\#\pi_0\)\\
\midrule
\((4_1,4_1!)\) & no & 30 & 116 & 236 & 2\\
\((4_1,4_1!)\) & yes & 30 & 58 & 118 & 1\\
\bottomrule
\end{tabular}
\caption{Mirror comparison for the supplied figure-eight seed and its reflected
mirror seed.}
\label{tab:mirror}
\end{table}

\paragraph{A seed-generated merge-scale check and certificate.}
The preceding computation shows that the two seed-generated BFACF components do
not merge at level \(N=30\).  Since BFACF moves change lattice length by
\(0,\pm2\), the next possible BFACF length level is \(N=32\).  The supplementary
script \texttt{merge\_check\_bfacf.py} performs a bidirectional search between two
supplied seeds.  For the figure-eight seed and its reflected mirror seed, the
level \(30\) check returns \texttt{connected=false} with status
\texttt{frontier\_exhausted}.  This status means that the two seed-generated
frontiers are exhausted under the length bound \(N\leq30\), not that a state or
time cap was reached.  With mirror symmetries identified, the level \(30\) check
returns \texttt{connected=true} with status \texttt{same\_canonical\_state}.

At the candidate merge level \(N=32\), the command recorded in the supplementary
README returned
\[
\texttt{connected=true},\qquad
\texttt{status=connecting\_state\_found}.
\]
The JSON summary records \(3036\) discovered states from the first search
frontier, \(10741\) discovered states from the second frontier, and \(11761\)
expanded states.  The first two numbers count discovered states in the two
frontiers at the stopping time, whereas the last number counts states actually
removed from a queue and expanded; they are therefore not expected to add up.  The
recorded runtime for this run was \(851.825\) seconds.  Runtime is
hardware-dependent; the supplementary README records the Python/Linux
configuration used for the archived run.

In addition to the connectivity summary, we extracted an explicit BFACF path
realizing the merge at level \(N=32\).  The path consists of \(21\) states and
\(20\) BFACF moves.  It passes through a \(32\)-edge connecting state \(\eta\): the
original \(30\)-edge seed \(\omega\) reaches \(\eta\) in \(5\) BFACF moves, while
the reflected mirror seed \(\omega!\) reaches the same state \(\eta\) in \(15\)
BFACF moves.  The edge-number sequence along the extracted full path is
\[
30,32,32,32,32,32,32,32,32,32,30,
32,32,32,32,32,32,30,30,30,30.
\]
Thus the maximum length along the path is \(32\), giving an explicit certificate
for the seed-generated merge at \(N=32\); see Figure~\ref{fig:figure8-merge-certificate}.

\begin{figure}[ht]
\centering
\includegraphics[width=\textwidth]{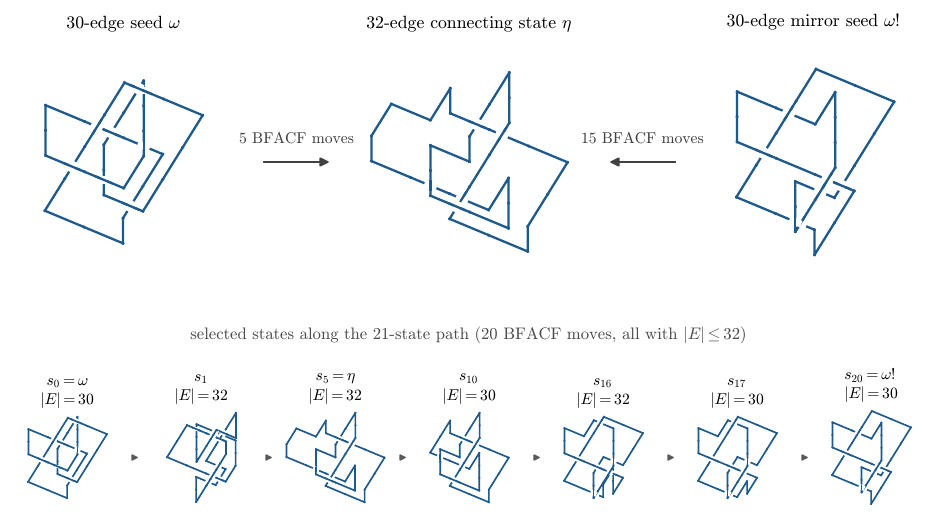}
\caption{An extracted BFACF path connecting the supplied \(30\)-edge
figure-eight seed \(\omega\) to its reflected mirror seed \(\omega!\) under the
length bound \(N=32\).  The two search branches meet at a \(32\)-edge state
\(\eta\).  The seed \(\omega\) reaches \(\eta\) in \(5\) BFACF moves, while the
mirror seed \(\omega!\) reaches \(\eta\) in \(15\) BFACF moves.  Hence the full
extracted path has \(21\) states and \(20\) moves, and gives an explicit
certificate for \(m_{\seed}^{\BFACF}(\omega,\omega!)=32\).  All states are drawn in one
fixed oblique projection of \(\mathbb Z^3\) with crossing information; dots mark
lattice vertices, and the lower row uses a common smaller scale.}
\label{fig:figure8-merge-certificate}
\end{figure}

Since the two seed-generated BFACF components are separated at \(N=30\), and
since there is no intermediate BFACF length level, we obtain the exact
seed-generated BFACF merge scale
\[
\boxed{m_{\seed}^{\BFACF}(\omega,\omega!)=32.}
\]

\begin{observation}[A length-two BFACF mirror barrier for the figure-eight seed]
For the supplied 30-edge simple cubic figure-eight seed \(\omega\) and its
reflected mirror seed \(\omega!\), with mirror symmetries not identified, the
seed-generated BFACF merge scale is
\[
 m_{\seed}^{\BFACF}(\omega,\omega!)=32.
\]
Equivalently, the two seeds are separated at the minimal level \(N=30\) and
become connected after allowing two additional lattice edges.
\end{observation}

The preceding seed certificate realizes the nontrivial height for the
verified published length-$30$ representatives.  The exhaustive bounded-component
calculation gives the following statement without presupposing that the external
population census is complete.

\begin{theorem}[BFACF merge tree on the reported length-$30$ $4_1$ census]
Let $S_4$ be the specified published collection of $152$ proper-cubic
equivalence classes of length-$30$ $4_1$ polygons.  Within the full
BFACF graph, these representatives meet exactly four distinct type-$0$
components of sizes $58,58,18,18$, and all four merge at cap $N=32$.
Consequently, for representatives $P,Q\in S_4$,
\[
 H_{\SC,\BFACF}(P,Q)=
 \begin{cases}
 30,& P,Q\text{ lie in the same birth-level component},\\
 32,& P,Q\text{ lie in distinct birth-level components},
 \end{cases}
\]
and the corresponding excess-length barrier is $0$ or $2$.
If $S_4$ exhausts every length-$30$ representative, then this is the
reported-layer merge tree of $4_1$.
\end{theorem}

\begin{figure}[ht]
\centering
\begin{tikzpicture}[x=1cm,y=1cm,font=\small]
\draw[->,>=stealth,thin] (-0.9,-0.2) -- (-0.9,2.75);
\node[rotate=90,font=\footnotesize] at (-1.75,1.0) {length cap $N$};
\foreach \y/\l in {0/30,2/32}{
  \draw[thin] (-0.97,\y) -- (-0.83,\y) node[left=3pt,font=\footnotesize] {$\l$};
  \draw[gray!45,densely dotted] (-0.83,\y) -- (6.3,\y);}
\foreach \x/\c in {0/C_1,1.6/C_2,3.8/C_3,5.4/C_4}{
  \draw[line width=1pt,blue!55!black] (\x,0) -- (\x,2);
  \fill[blue!55!black] (\x,0) circle (1.8pt);}
\draw[line width=1pt,blue!55!black] (0,2) -- (5.4,2);
\draw[line width=1pt,blue!55!black] (2.7,2) -- (2.7,2.45);
\node[font=\footnotesize,align=center,above] at (2.7,2.45) {one component: $19{,}618$ states,\\ including all $152$ minimal ones};
\foreach \x/\c/\s in {0/C_1/58,1.6/C_2/58,3.8/C_3/18,5.4/C_4/18}{
  \node[below,font=\footnotesize,align=center] at (\x,-0.08) {$\c$\\ $\s$ states};}
\draw[<->,>=stealth,gray] (0.25,0.55) -- node[above,font=\scriptsize,text=black] {reflection} (1.35,0.55);
\draw[<->,>=stealth,gray] (4.05,0.55) -- node[above,font=\scriptsize,text=black] {reflection} (5.15,0.55);
\end{tikzpicture}
\caption{BFACF merge tree on the reported census of the minimal simple-cubic $4_1$ layer.
The four type-$0$ components at $N=30$ have sizes $58,58,18,18$ and merge
simultaneously at $N=32$.  The paired $58$- and $18$-vertex components are
exchanged by reflection.  Since simple cubic polygons have even length, no
level lies strictly between $30$ and $32$.}
\label{fig:figure8-complete-merge-tree}
\end{figure}

This theorem is global for the published length-$30$ $4_1$ collection under the specified
BFACF move system and quotient.  A different lattice or a different complete
local-move system may produce a different minimax hierarchy.

\FloatBarrier
\subsection{Symmetry breaking at minimal lattice length}
\label{subsec:fourfold-symmetry}

The mirror pairing of the four birth-level components has an immediate
consequence for geometric symmetries of minimal representatives.  For
comparison, there is also a simple length-divisibility obstruction to
fourfold rotoreflection symmetry.

\begin{proposition}[Fourfold divisibility]
\label{prop:fourfold-divisibility}
Let $P$ be a simple-cubic lattice knot polygon invariant under the
order-four rotoreflection
\[
 S(x,y,z)=(-y,x,-z).
\]
Then the number of edges of $P$ is divisible by four.
\end{proposition}

\begin{proof}
The square $S^2$ is the half-turn about the $z$-axis.  Its restriction
to the topological circle $P$ is orientation-preserving, since it is
the square of the restriction of $S$.  If $S^2$ fixed a point of $P$,
then, as a finite-order orientation-preserving homeomorphism of a circle
with a fixed point, it would be the identity on $P$.  This would place
all of $P$ on the rotation axis, an impossibility for a knot polygon.
Thus $P$ is disjoint from the fixed axis of $S^2$.

If a lattice edge of $P$ were preserved setwise by $S^2$, its midpoint
would lie on that fixed axis, a contradiction.  Hence every edge has
an orbit of exactly four elements under $S$, proving the claim.
\end{proof}

Since $n_{\SC}(4_1)=30$, \cref{prop:fourfold-divisibility}
implies that an $S$-invariant simple-cubic representative of $4_1$ has
at least $32$ edges; no assertion of existence at $32$ edges is made.
The computation for the published shortest-layer collection yields a stronger conclusion.

\begin{corollary}[No improper symmetry at minimal length]
\label{cor:no-improper-minimal-figure-eight}
No $30$-edge simple-cubic representative of $4_1$ is preserved by an
orientation-reversing isometry of the simple cubic lattice.
\end{corollary}

\begin{proof}
Fix a lattice reflection $r$.  If an orientation-reversing lattice
isometry $g$ preserved a minimal polygon $P$, then $r\circ g^{-1}$
would be orientation-preserving and
\[
  r(P)=(r\circ g^{-1})(P).
\]
Consequently $P$ and $r(P)$ would define the same vertex in the
proper-cubic quotient, and hence lie in the same birth-level BFACF
component.  This contradicts the mirror pairing
$C_1\leftrightarrow C_2$, $C_3\leftrightarrow C_4$ established above.
\end{proof}

Thus amphichirality of the knot type does not entail an
orientation-reversing geometric symmetry of a minimal lattice
representative.  This discrete symmetry breaking is distinct from,
and does not establish, its analogue for smooth ropelength minimizers.

\subsection{The amphichiral knot \texorpdfstring{\(6_3\)}{6-3} and its mirror}

The knot \(6_3\) is also amphichiral.  Its minimal simple cubic lattice length is
\[
 n_{\SC}(6_3)=40,
\]
with published minimal population and symmetry-class counts
\[
 P_{40}(6_3)=3552,\qquad S_{40}(6_3)=148
\]
\cite{MinimalKnotsRechnitzer}, so the birth layer of \(6_3\) has almost exactly
the same enumeration size as that of \(4_1\).
For this experiment we converted the 40-vertex coordinate list on Rechnitzer's
published \emph{Minimal \(6_3\) knots} data page into
\texttt{six3\_seed.json}; the original coordinate file, its checksum, and a
provenance note are included in the supplementary archive
\cite{MinimalKnotsRechnitzer,Rechnitzer63Data}.  Let \(\sigma\) denote this
40-edge seed and let \(\sigma!\) be the reflected seed obtained from
\[
 (x,y,z)\longmapsto(-x,y,z).
\]
The script \texttt{verify\_seed\_pair.py} checks that both inputs are valid
40-edge simple cubic polygons, are distinct under proper cubic symmetry, and
become equal when reflections are allowed.  In addition, the script
\texttt{verify\_knot\_type.py} verifies the knot type itself: it computes the
Alexander polynomial of a lattice polygon from a generic planar projection,
and applied to \(\sigma\) and \(\sigma!\) it returns
\[
\Delta(t)\doteq t^{2}-3t+5-3t^{-1}+t^{-2},
\]
the Alexander polynomial of \(6_3\), which distinguishes \(6_3\) from every
other knot with at most six crossings.  The same script reproduces the
expected Alexander polynomials of all previously supplied seeds (\(3_1\),
\(4_1\), \(4_1!\), \(5_1\), and the composite examples), which validates the
verifier itself; the outputs are archived as JSON verification files.

The published shortest-layer collection gives substantially more information than
the initial two-seed comparison.  At $N=40$ the $148$ minimal classes split into
twelve type-$0$ BFACF components with sizes
\[
39,39,10,10,9,9,7,7,6,6,3,3.
\]
They occur in six mirror pairs.  We label the birth components
\[
A_{39},A_{10},A_9,A_7,A_6,A_3,
\qquad
B_{39},B_{10},B_9,B_7,B_6,B_3,
\]
where the subscript is the component size and the labels are chosen so that the
reflection
\[
\mu(x,y,z)=(-x,y,z)
\]
induces
\[
\mu(A_r)=B_r,\qquad \mu(B_r)=A_r
\qquad
(r\in\{39,10,9,7,6,3\}).
\]
Thus the mirror pairing is
\[
A_{39}\leftrightarrow B_{39},\quad
A_{10}\leftrightarrow B_{10},\quad
A_9\leftrightarrow B_9,\quad
A_7\leftrightarrow B_7,\quad
A_6\leftrightarrow B_6,\quad
A_3\leftrightarrow B_3.
\]
The supplied seed $\sigma$ lies in $A_7$, after relabelling if necessary, and
its reflected seed $\sigma!$ lies in $B_7$.  Notice that $6_3$ is amphichiral,
so $\sigma$ and $\sigma!$ represent the same knot type even though the
mirror-sensitive bounded move graph need not connect them at a given length
cap.

At $N=42$, exhaustive search from $\sigma$ yields exactly $12337$ bounded
canonical states, of which $74$ have length $40$ and $12263$ have length $42$.
The $74$ minimal states are precisely
\[
A_{39}\sqcup A_{10}\sqcup A_9\sqcup A_7\sqcup A_6\sqcup A_3.
\]
We denote this level-$42$ component by $A$.  An independent exhaustive search
from $\sigma!$ gives a second $12337$-state component $B$ whose minimal states
are precisely
\[
B_{39}\sqcup B_{10}\sqcup B_9\sqcup B_7\sqcup B_6\sqcup B_3.
\]
The two bounded state sets are disjoint, and reflection is an automorphism of
the BFACF graph preserving lattice length, so
\[
\mu(A)=B,\qquad \mu(B)=A.
\]
Therefore the twelve birth-level components have merged into exactly two
mirror-image level-$42$ components $A$ and $B$.

At \(N=44\), a deterministic C++17 implementation of the same BFACF moves and
canonicalization found a connecting state.  The search recorded \(1601979\)
discovered states from the first side, \(1028207\) from the second side, and
\(2555900\) expanded states before the intersection was found.  The recorded
runtime was \(347.037\) seconds, which is hardware-dependent.  The C++
accelerator was cross-checked against the Python implementation at \(N=40\) and
\(N=42\), and the final path was independently checked by
\texttt{verify\_merge\_certificate.py} using the Python neighbor generator.

The extracted certificate consists of \(153\) states and \(152\) BFACF moves.
The two branches meet at a 42-edge state \(\eta\): the seed \(\sigma\) reaches
\(\eta\) in \(116\) moves, while \(\sigma!\) reaches it in \(36\) moves.  Some
intermediate states have \(44\) edges, and no state has more than \(44\) edges.
Figure~\ref{fig:six3-merge-certificate} displays the endpoints, the meeting
state, and selected states along the archived path.

\begin{table}[ht]
\centering
\small
\begin{tabular}{clp{5.0cm}l}
\toprule
level & result & exact evidence & principal size data\\
\midrule
40 & separated & complete two-seed graph & $14$ vertices; $7+7$\\
42 & separated & complete graph; frontiers exhausted & $24674$ vertices; $12337+12337$\\
44 & connected & verified explicit path & $153$ states; $152$ moves\\
\bottomrule
\end{tabular}
\caption{Seed-generated BFACF mirror analysis for the supplied minimal
$6_3$ seed and its reflected mirror.  Reflections are not identified.}
\label{tab:six3-mirror}
\end{table}

\begin{figure}[ht]
\centering
\includegraphics[width=\textwidth]{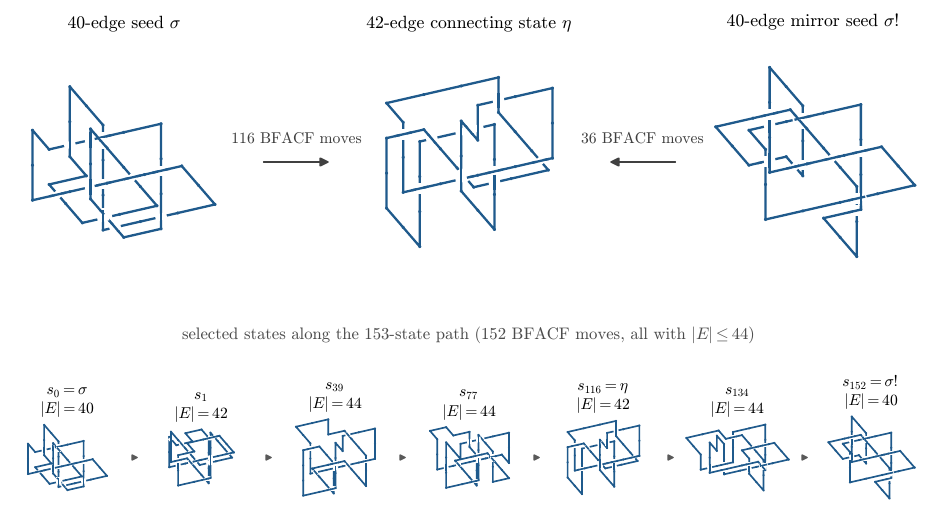}
\caption{An extracted BFACF path connecting the supplied 40-edge $6_3$ seed
$\sigma$ to its reflected mirror seed $\sigma!$ under the length bound $N=44$.
The two search branches meet at a 42-edge state $\eta$.  The seed reaches
$\eta$ in $116$ moves and the mirror seed reaches it in $36$ moves.  The full
certificate has $153$ states and $152$ moves, and its maximum edge number is
$44$.  Drawing conventions are as in \cref{fig:figure8-merge-certificate}.}
\label{fig:six3-merge-certificate}
\end{figure}

Because BFACF changes length by \(0\) or \(\pm2\), the exact separation at
\(N=42\) and the verified path at \(N=44\) imply
\[
 \boxed{m_{\seed}^{\BFACF}(\sigma,\sigma!)=44.}
\]

\begin{observation}[A length-four BFACF mirror barrier for the supplied $6_3$ seed]
For the supplied 40-edge minimal simple cubic $6_3$ seed \(\sigma\) and its
reflected mirror \(\sigma!\), with mirror symmetries not identified, the
seed-generated BFACF merge scale is
\[
 m_{\seed}^{\BFACF}(\sigma,\sigma!)=44.
\]
Equivalently, the two seeds remain separated after allowing two extra edges and
become connected after allowing four extra edges.
\end{observation}

The seed certificate realizes the final merge among the reported length-$40$
representatives.  Combining the exhaustive $N=40$ and $N=42$ component
calculations with the verified $N=44$ path gives a statement for that
collection, independent of whether the reported census exhausts the true
minimal layer.

\begin{theorem}[BFACF merge tree on the reported length-$40$ $6_3$ census]
Let $S_6$ be the specified published collection of $148$ proper-cubic
equivalence classes of length-$40$ $6_3$ polygons.  The representatives
in $S_6$ meet twelve distinct type-$0$ BFACF components at length $40$,
with sizes
\[
39,39,10,10,9,9,7,7,6,6,3,3.
\]
With the notation above, at $N=42$ the six components $A_r$ merge into one
component $A$ and the six components $B_r$ merge into one component $B$.  The
reflection $\mu$ exchanges $A_r$ with $B_r$ and exchanges the two level-$42$
components $A$ and $B$.  Each of $A$ and $B$ contains $74$ minimal classes and
$12337$ bounded states, and they are disjoint at $N=42$.  These two components
merge at $N=44$.  Hence the complete component hierarchy is
\[
12\longrightarrow 2\longrightarrow 1
\qquad (N=40,42,44).
\]
For minimal representatives, the excess-length barrier is $0$ within a
birth-level component, $2$ between distinct birth-level components among the
$A_r$'s or among the $B_r$'s, and $4$ between any $A_r$ and any $B_s$.  In
particular, each mirror pair $A_r,B_r$ is separated through $N=42$ and merges
only at $N=44$.  If $40$ is the true minimum length for $6_3$ and
$S_6$ exhausts that layer, then this is the entire reported shortest-layer
merge tree.
\end{theorem}

\begin{figure}[ht]
\centering
\begin{tikzpicture}[x=1cm,y=1cm,font=\small]
\draw[->,>=stealth,thin] (-0.9,-0.2) -- (-0.9,3.9);
\node[rotate=90,font=\footnotesize] at (-1.75,1.6) {length cap $N$};
\foreach \y/\l in {0/40,1.6/42,3.2/44}{
  \draw[thin] (-0.97,\y) -- (-0.83,\y) node[left=3pt,font=\footnotesize] {$\l$};
  \draw[gray!45,densely dotted] (-0.83,\y) -- (12.3,\y);}
\draw[gray,dashed] (5.95,-0.35) -- (5.95,2.9);
\node[font=\scriptsize,gray,fill=white,inner sep=1pt] at (5.95,-0.5) {$\mu$};
\foreach \x in {0,0.95,1.9,2.85,3.8,4.75}{\draw[line width=1pt,green!40!black] (\x,0) -- (\x,1.6); \fill[green!40!black] (\x,0) circle (1.7pt);}
\draw[line width=1pt,green!40!black] (0,1.6) -- (4.75,1.6);
\draw[line width=1pt,green!40!black] (2.375,1.6) -- (2.375,3.2);
\foreach \x in {7.15,8.1,9.05,10.0,10.95,11.9}{\draw[line width=1pt,red!60!black] (\x,0) -- (\x,1.6); \fill[red!60!black] (\x,0) circle (1.7pt);}
\draw[line width=1pt,red!60!black] (7.15,1.6) -- (11.9,1.6);
\draw[line width=1pt,red!60!black] (9.525,1.6) -- (9.525,3.2);
\draw[line width=1pt,blue!55!black] (2.375,3.2) -- (9.525,3.2);
\draw[line width=1pt,blue!55!black] (5.95,3.2) -- (5.95,3.55);
\node[font=\footnotesize,above] at (5.95,3.55) {one component};
\node[font=\footnotesize,align=left,anchor=west,fill=white,inner sep=1.5pt] at (2.5,2.35) {$A$: $74$ minimal classes,\\ $12{,}337$ states};
\node[font=\footnotesize,align=left,anchor=west,fill=white,inner sep=1.5pt] at (9.65,2.35) {$B=\mu(A)$: $74$ minimal\\ classes, $12{,}337$ states};
\foreach \x/\r in {0/39,0.95/10,1.9/9,2.85/7,3.8/6,4.75/3}{\node[below,font=\footnotesize,align=center] at (\x,-0.08) {$A_{\r}$\\[-1pt]\scriptsize $(\r)$};}
\foreach \x/\r in {7.15/3,8.1/6,9.05/7,10.0/9,10.95/10,11.9/39}{\node[below,font=\footnotesize,align=center] at (\x,-0.08) {$B_{\r}$\\[-1pt]\scriptsize $(\r)$};}
\end{tikzpicture}
\caption{BFACF merge tree on the reported census of the minimal simple-cubic $6_3$ layer.
At $N=40$ the twelve type-$0$ components form the six mirror pairs
$A_r\leftrightarrow B_r$ for $r=39,10,9,7,6,3$.  At $N=42$ the six $A_r$
components merge to $A$ and the six $B_r$ components merge to
$B=\mu(A)$; the two mirror-image branches merge at $N=44$.  Numbers in
parentheses are numbers of minimal classes, and the dashed line indicates the
reflection $\mu$.}
\label{fig:six3-complete-merge-tree}
\end{figure}

This hierarchy is global for the published length-$40$ $6_3$ collection under the
specified BFACF move system and quotient.  Its values are not asserted to be
invariant under a change of lattice or move system.

\paragraph{An independent second seed.}
To reduce dependence on any single seed, on the published coordinate data, and
on any single implementation, the archive contains a second minimal 40-edge
\(6_3\) seed \(\sigma'\) constructed without external coordinate data.  The
script \texttt{make\_six3\_seed.py} builds a \(156\)-edge cubical closed-braid
representative of \(6_3\) from the standard \(3\)-strand braid word
\(\sigma_1^{-1}\sigma_1^{-1}\sigma_2\sigma_1^{-1}\sigma_2\sigma_2\) and
shortens it to length \(40\) by a seed-recorded stochastic BFACF descent;
every intermediate step is a valid BFACF move, so the knot type is preserved,
and \texttt{verify\_knot\_type.py} confirms the \(6_3\) Alexander polynomial
for the braid polygon, the seed, and its reflected mirror.  The pair
\((\sigma',\sigma'!)\) repeats the entire analysis with an independent
pure-Python search implementation (\texttt{merge\_check\_bfacf\_fast.py},
whose canonicalization is validated against the reference script by a built-in
self-test and which reproduces the \(4_1\) results): the pair is separated at
\(N=40\) and at \(N=42\), and connected at \(N=44\) by an archived and
independently verified \(157\)-state, \(156\)-move certificate, giving
\(m_{\seed}^{\BFACF}(\sigma',\sigma'!)=44\) as well.  The exhausted level-42
component of \(\sigma'\) contains exactly \(12337\) states, matching the
complete enumeration above, and this is no coincidence: an archived and
verified \(15\)-move BFACF path at \(N\leq 42\) connects \(\sigma'!\) to
\(\sigma\), so \(\sigma'\) lies in the level-42 component of \(\sigma!\) and
\(\sigma'!\) in that of \(\sigma\).  Thus two independently constructed
minimal seeds, processed by two independent implementations (C++ and Python),
certify the same merge event at \(N=44\) from opposite sides of the same pair
of level-42 components.

\subsection{The five-crossing knots}

The two prime knots with five crossings have the same crossing number but occupy
different levels in the simple cubic lattice filtration.  Published enumeration
data give
\[
 n_{\SC}(5_1)=34,\qquad n_{\SC}(5_2)=36,
\]
and minimal populations
\[
 P_{34}(5_1)=6672,
 \qquad
 P_{36}(5_2)=114912
\]
\cite{MinimalKnotsRechnitzer}.  Thus the birth level already distinguishes the
two five-crossing prime knots in the simple cubic lattice filtration.

For \(5_1\) we used a 34-edge minimal simple cubic lattice representative.  The
birth-level seed-generated graph is very small: it has \(5\) vertices, \(4\)
edges, and one component.  This indicates that the supplied minimal \(5_1\)
seed is locally quite rigid at the birth level: only a few length-preserving
BFACF corner changes are available before one leaves the length-34 layer.
Allowing length up to \(36\) and stopping after \(1000\) vertices gives a
partial exploration with \(2207\) edges and one component; see
Table~\ref{tab:seed-basic}.

For \(5_2\), the supplied test seed is not a minimal 36-edge representative.
Instead, it is a cubical closed-braid seed of length \(238\).  We now treat this
calculation only as an implementation stress test, not as evidence about the
birth layer of \(5_2\).  With length bound \(238\) and cap \(1000\), the explored
subgraph has \(1000\) vertices, \(1043\) edges, and one component.  A genuine
birth-level experiment for \(5_2\) would require a curated list of minimal
36-edge representatives and should replace this stress test in a future version.

\subsection{Comparison cases through eight crossings}
\label{subsec:low-crossing-comparison}

The low-crossing examples suggest three separate comparisons: equal
reported birth levels with different knot types, rigidity at a fixed
level, and the action of reflection.  We distinguish literature data
from the coordinate certificates supplied here.

\subsubsection{Reported lengths and candidate comparisons}

Table~\ref{tab:low-crossing-lengths} extracts useful comparison cases
from \cite[Table~1]{MinimalKnotsRechnitzer}.  We write
$n_{\mathrm{rep}}$ for the reported minimal length, to avoid treating
a sampled upper bound as a newly proved minimum.

\begin{table}[ht]
\centering
\begin{tabular}{cl}
\toprule
$n_{\mathrm{rep}}$ & Selected knot types\\
\midrule
40 & $6_1,6_2,6_3$\\
42 & $8_{19}$\\
44 & $7_1,7_3,7_4,7_7,8_{20}$\\
46 & $7_2,7_5,7_6,8_{21}$\\
48 & $8_3,8_7$\\
50 & $8_9$\\
52 & $8_{12},8_{17},8_{18}$\\
\bottomrule
\end{tabular}
\caption{Selected published simple-cubic lengths through eight crossings.
These are comparison data, not additional complete merge-tree computations.}
\label{tab:low-crossing-lengths}
\end{table}

The common reported level $40$ of $6_1,6_2,6_3$ makes them a natural
test of information retained by merge data beyond birth alone.
The torus knot $8_{19}=T(3,4)$ \cite{KnotAtlasEight} is reported at
$42$, below the reported $44$ or $46$ of every seven-crossing prime knot.
Thus even the reported birth ordering is not crossing-number ordering.
Neither observation determines a merge height; doing so requires
component closure tests and connecting certificates of the kind used above.

\subsubsection{Singleton components and a zero mirror barrier}

The published coordinate files for $8_3$ and $8_7$
\cite{RechnitzerEightData} give especially small components to check.
Let $P_{83}$ and $P_{87}$ denote the polygons in the archived
\path{SC83.dat} and \path{SC87.dat}, with the knot-type identifications
of that data source.

\begin{proposition}[Two rigid representatives at length 48]
\label{prop:eight-rigid-seeds}
Each of $[P_{83}]$ and $[P_{87}]$ is a singleton connected component of
the corresponding graph $G_{48}^{\SC,\BFACF}(K)$ in the proper-cubic,
unoriented quotient.  For $P_{83}$, every spatial lattice reflection $r$
satisfies
\[
 [r(P_{83})]=[P_{83}],\qquad
 H_{\SC,\BFACF}(P_{83},r(P_{83}))=48,
 \qquad B_{\SC,\BFACF}(P_{83},r(P_{83}))=0.
\]
These statements concern the specified representatives and do not
assume that they exhaust their full length-$48$ layers.
\end{proposition}
\begin{proof}
The coordinate verifier checks that both polygons have $48$ unit edges,
are self-avoiding, and admit no shortening move.  Every legal neutral
neighbor canonicalizes to the original polygon.  Thus their capped
components consist of a single vertex, after loops are omitted.

In the original coordinates, the edge set of $P_{83}$ is invariant under
\[
 T(x,y,z)=(3-y,x,5-z).
\]
This is an orientation-reversing lattice isometry of order four.
For any lattice reflection $r$, the map $r\circ T^{-1}$ is
orientation-preserving and sends $P_{83}$ to $r(P_{83})$.
The two therefore represent the same vertex.  A constant quotient path
gives height $48$, and no path can have a smaller peak than its endpoints.
\end{proof}

The proper-rotation orbit sizes of these representatives are $12$ and
$24$, respectively.  The two rows of each downloaded symmetry file
canonicalize to one vertex under our conventions; raw row counts must
therefore not be interpreted as vertex counts.  The verifier records
the signed-permutation symmetries and the digit-convention conversion.
It checks the assertions above directly from the supplied coordinates,
without proving global minimality or a complete population count.

This example separates two phenomena.  The $8_3$ representative is
rigid under the length cap, yet its mirror is already the same vertex.
By contrast, the published shortest-layer collections of $4_1$ and $6_3$ have
mirror barriers $2$ and $4$.  A small component alone is consequently
not evidence for a large mirror barrier.  The symmetry $T$ is a
translated order-four rotoreflection, consistent with the divisibility
argument of \cref{prop:fourfold-divisibility} and $4\mid48$.

\subsubsection{Amphichirality and the orientation convention}

The further eight-crossing amphichiral examples $8_9,8_{12},8_{17},8_{18}$
provide candidates for extending the mirror-branch computation
\cite{KnotAtlasEight}.  Their topological amphichirality guarantees
eventual merging in the present unoriented BFACF model, but gives no
finite-cap value and does not imply a mirror-fixed lattice representative.

The case $8_{17}$ also tests the choice of state space.  It is
noninvertible and negatively amphichiral: as an oriented knot,
$K\cong\overline K^{-1}$, whereas $K\not\cong K^{-1}$ and
$K\not\cong\overline K$ \cite{KnotAtlasEight}.  Our quotient forgets
traversal orientation, so its mirror still lies in the same unoriented
knot type and BFACF ergodicity applies.  In an oriented version, a
polygon and its mirror with the induced knot orientation belong to different
oriented knot types; a mirror-reverse comparison is required instead.
This distinction should be fixed before interpreting any proposed
$8_{17}$ merge tree.

\FloatBarrier

\subsection{Composite trefoil examples as a sanity check}

We also made a very small sanity check using two non-minimal composite trefoil
seeds: the granny knot \(3_1\#3_1\) and the square knot \(3_1\#3_1!\).  These
knots are topologically distinct, for example by chirality-sensitive polynomial
invariants \cite{Jones}.  The purpose of the present computation is not to prove
or rediscover this distinction.  It only checks that the implementation keeps two
known different seed types separated in a tiny capped exploration.

The seeds have length \(66\).  Running the two seeds separately with a cap of
\(20\) vertices gives one component in each explored subgraph.  Running the two
seeds together, without mirror identification and with a cap of \(40\) vertices,
gives \(40\) vertices, \(40\) edges, and two components of sizes \(29\) and
\(11\).  Because the cap is extremely small relative to the size expected at this
length, this check is not used as scientific evidence for a merge-scale claim.
It is included only as a reproducibility and consistency test.

\subsection{Computational interpretation: a minimax reconfiguration height}
\label{subsec:minimax-height}

For two lattice polygons \(P,Q\) connected by a chosen move system \(\mathcal M\),
define the minimax reconfiguration height
\[
 H_{\mathcal L,\mathcal M}(P,Q)
 =\min_{\Gamma:P\leadsto Q}\ \max_{R\in\Gamma}\ell(R),
\]
where the minimum runs over finite \(\mathcal M\)-paths \(\Gamma\) from \(P\) to
\(Q\).  The associated excess-length barrier is
\[
 B_{\mathcal L,\mathcal M}(P,Q)
 =H_{\mathcal L,\mathcal M}(P,Q)-\max\{\ell(P),\ell(Q)\}.
\]
For standard simple-cubic BFACF moves, ergodicity gives
\(H_{\SC,\BFACF}(P,Q)<\infty\) for polygons of the same knot type
\cite{BFACF}.  If \(P,Q\) lie in the initial lattice layer, this minimax height is
exactly the first filtered level at which their initial components merge.

Thus the computations above may be stated succinctly as
\[
 B_{\SC,\BFACF}(\omega,\omega!)=2,
 \qquad
 B_{\SC,\BFACF}(\sigma,\sigma!)=4.
\]
For $P\in D_4$ and any minimal trefoil $Q$, the additional computation
in \cref{thm:trefoil26-heights} gives $B_{\SC,\BFACF}(P,Q)=4$;
for the other five $D_i$ it gives $B_{\SC,\BFACF}(P,Q)=2$.
This formulation makes the relation with earlier work transparent.  Minimal
lattice enumeration and type-0 BFACF classes describe the bottom layer
\cite{Scharein2009}; the present computation measures how far above that layer
one must climb before prescribed classes can be reconfigured into one another by
the full BFACF move set.

\section{Controlled comparison with geometric models}
\label{sec:controlled-comparison}
This section states the precise extra information needed to compare a
finite move model with a geometric deformation space. It constructs no
unconditional lattice-to-continuum equivalence.

\begin{proposition}[Components under controlled simulation]
Let $X,Y$ be resource-filtered deformation systems, with paths closed under
restriction, reversal and concatenation. Suppose maps $F:X\to Y$, $G:Y\to X$
increase resource by at most $\varepsilon$, simulate every level-$a$ path
by a level-$(a+\varepsilon)$ path, and admit connectors from $x$ to $GFx$
and $y$ to $FGy$ at levels at most $r_X(x)+2\varepsilon$ and
$r_Y(y)+2\varepsilon$. Then their component functors are
$\varepsilon$-interleaved. Moreover,
\[
 \mu_Y(Fx,Fx')\le\mu_X(x,x')+\varepsilon,
\]
\[
 \mu_X(x,x')\le\max\{\mu_Y(Fx,Fx')+\varepsilon,
 r_X(x)+2\varepsilon,r_X(x')+2\varepsilon\}.
\]
\end{proposition}
\begin{proof}
Path simulation makes the two component maps well defined at shifted levels.
The connectors identify each composite with the inclusion shifted by
$2\varepsilon$, which proves the interleaving identities. Simulate paths
whose peak approaches the first merge infimum to obtain the first inequality.
For the second, apply $G$ to such paths in $Y$ and append the two connectors.
Their peak is bounded by the displayed maximum; take the infimum.
\end{proof}
A graph is such a system when each edge carries a verified resource budget
and a path's budget is the maximum of all its vertex and edge budgets.
In finite windows all shifted levels used in the proposition must exist.
No higher-homotopy conclusion follows without continuous realization maps
and compatible homotopies.

\begin{corollary}[Conditional excess transfer]
For a corresponding endpoint pair, suppose
$|\mu_X-\mu_D|\le\delta_m$ and each geometric endpoint resource is at most
the discrete endpoint resource plus $\delta_r$. Then
\[
 \mu_X-\max(r_X(x),r_X(x'))
 \ge\mu_D-\max(r_D(d),r_D(d'))-\delta_m-\delta_r.
\]
\end{corollary}
\begin{proof}
Subtract the endpoint maximum bound from the lower merge bound.
\end{proof}
A forward smoothing of discrete paths supplies an upper bound on the
continuum merge scale. It alone cannot transfer a positive discrete barrier
to a continuum lower bound.

\subsection{Two obstructions to automatic convergence}
The Euclidean reach of a polygon with genuine corners is zero. A polygonal
thickness convention is therefore not the reach used for the smooth model;
rounding must be specified and estimated. At mesh $h$, physical lattice
length is $hN$. Coordinate-axis paths from $(0,0,0)$ to $(1,1,1)$ have
length at least $3$, whereas the straight segment has length $\sqrt3$.
This anisotropy does not disappear merely by refining the mesh.

Even convergence of static minimization problems need not control paths.
On $[-1,1]$, let $f_n(t)=\max(1-n|t|,0)$. These nonnegative functions
Gamma-converge to zero: away from zero use the constant recovery sequence,
and at zero use $t_n=1/n$. Yet every continuous path from $-1$ to $1$
crosses zero, hence has $f_n$ peak $1$, whereas the limit peak is $0$.
A pathwise comparison theorem is indispensable for merge convergence.

\subsection{Certificate scope}
Three forms of evidence must be kept distinct. Exhaustive closure of every
minimal-layer seed certifies its component partition only if completeness
of the minimal seed list is also established. Exhausting an entire capped
reachable set certifies all states in that component at that cap. An
explicit finite move sequence certifies a positive connection and its peak,
but cannot by itself certify absence of a lower-budget connection.
The attached prototype enumerates capped reachable sets; its output path is
configurable. The numerical results stated earlier retain their original
data provenance. A diagram displaying selected states is not a replacement
for a machine-readable full move certificate.

Finally, retain the component partitions themselves, not only the $H_0$
barcode. Linear persistence records births and deaths but can forget which
clusters merge. The merge-tree data in this paper include that incidence.

\subsection{Adding an observed resource}
If an observable obeys $|\Phi_Y(Fx)-\Phi_X(x)|\le\delta$ for every state,
then a level-$(a,c)$ path satisfying $r_X\le a$, $\Phi_X\le c$ maps to a
path satisfying $r_Y\le a+\varepsilon$, $\Phi_Y\le c+\delta$, provided
the simulator satisfies these inequalities along the entire path.
For a reverse interleaving the connectors must obey the observed-resource
bounds as well. Endpoint estimates alone do not bound a resource in the
interior of a deformation. Likewise, a proposed statistic descends to a
knot-type invariant only when it is constant on the full allowed equivalence
relation; training or testing on a finite sample does not prove that property.

\section{Further questions}
\label{sec:further}

Throughout the first problems, use the standard simple-cubic BFACF model and
the proper-cubic quotient of \cref{subsec:eventual-connectivity}.

\begin{remark}[Existence questions settled in this revision]
\label{rem:settled-existence}
An earlier version of this section posed four existence questions.  Was
eventual connectivity false for every nontrivial knot type?  Were strong
birth levels unbounded for knot types other than the unknot and the trefoil?
Did every knot type admit a lattice Gordian configuration?  Could the frozen
Hamiltonian-block construction be adapted to arbitrary knot patterns without
creating new nonincreasing moves at the interfaces?  All four are answered
by \cref{thm:frozen-every-knot} and \cref{cor:every-knot}.  For
\emph{every} knot type $K$, of either chirality, the capped graphs
$G_N(K)$ are not eventually connected: $G_{125m}(K)$ is disconnected for
every even $m\geq n_K$.  Strong local-minimum components are born at all
lengths $125m$ with $m$ even and $m\ge n_K$.  At each such length, their
number is at least the number of proper-cubic classes of type $K$ and
length $m$.  Every knot type has BFACF-frozen, hence lattice Gordian,
representatives.  The adaptation is the block substitution of
\cref{def:block-substitution}: every vertex becomes a Hamiltonian
$5\times5\times5$ block, and face-centred ports with normal end edges make
all interfaces inert.  Side length five is forced by the parity
obstruction explained in \cref{subsec:frozen-all}.
\end{remark}

\begin{problem}[Unbounded escape excess]
\label{prob:unbounded-excess}
For which fixed knot types $K$ does
\[
 \sup_{P\text{ of type }K}\bigl(h_K(P)-\ell(P)\bigr)=\infty?
\]
\end{problem}

\Cref{cor:excess-births} shows that unbounded excess would imply unbounded
birth levels.  The latter now hold for every knot type, so this
implication gives no further restriction.  All frozen families
constructed here satisfy only the uniform lower bound $e_K\geq2$.
Freezing is a statement about the first move; it does not control
how quickly a polygon can shorten after one $+2$ move.  The exact
barrier $H(P_1,\square)-64=2$ for the $64$-edge block shows that a
fully packed configuration can have the least possible positive excess.
A proof of unbounded excess would need an obstruction that persists
along every path with bounded excess, not merely at the starting
configuration.

\begin{problem}[Excess of block substitutions]
\label{prob:excess-substitution}
Compute $e_K(\Sigma(Q))$ for small $Q$, or prove that
$e_K(\Sigma(Q))=2$ for every polygon $Q$.  Is $e_K(\Sigma^k(Q))$ bounded
as $k\to\infty$ (\cref{rem:iterated-substitution})?
\end{problem}

\begin{problem}[Shortest nonminimal frozen representatives]
\label{prob:shortest-frozen}
Let $f(K)$ be the least length $\ell>n_K$ of a simple-cubic polygon of type
$K$ with no legal BFACF move of length change $0$ or $-2$.  By
\cref{thm:frozen-every-knot},
\[
 f(K)\leq125\,n_K\quad\text{for every }K,\qquad
 f(0_1)\leq64,\qquad f(3_1)\leq88.
\]
Determine $f(K)$ for small knots.  Is $f(K)/n_K$ bounded by a constant
much smaller than $125$?  The analogous question without the singleton
requirement asks for the least nonminimal strong birth length of $K$.
For the trefoil of the fixed chirality, that length is $26$
(\cref{subsec:trefoil-births}).
\end{problem}

The bounds $f(0_1)\le64$ and $f(3_1)\le88$ are witnessed by $P_1$ of
\cref{subsec:frozen-unlinks} and $R_1$ of \cref{subsec:frozen-trefoil}.
The restriction $\ell>n_K$ is essential: the minimal $24$-edge trefoil $Q$
of \cref{subsec:frozen-trefoil} is itself frozen.

\begin{problem}[Quantitative finite-input merging]
\label{prob:quantitative-merging}
For a fixed minimal representative $Q_*$, determine effective upper and
lower bounds for $M_K(L;Q_*)$ as a function of $L$ and $K$.
\end{problem}

Its finiteness follows from classical BFACF connectivity and finite
sublevels.  \Cref{cor:every-knot} gives the trivial-looking but uniform
lower bound $M_K(L;Q_*)\geq125m+2$ whenever $125m\leq L$ with $m\geq n_K$
even.  Sharp bounds and certified computations would carry more knot-type
information.

The following criterion is no longer needed for existence of strong
minima, but it isolates a different, \emph{confined} mechanism for them.

\begin{lemma}[Coordinate confinement under nonincreasing moves]
\label{lem:coordinate-box}
On actual coordinate representatives, a neutral BFACF move preserves
the axis-aligned bounding box of a polygon.  A shortening move cannot
enlarge that box.  In particular, the sorted coordinate-span triple is
invariant under neutral moves in the proper-cubic quotient.
\end{lemma}
\begin{proof}
A neutral move replaces the corner $V$ of a unit plaquette with its
opposite corner $W$, while leaving the two endpoints $A,B$ in the polygon.
For every coordinate $j$, both $V_j$ and $W_j$ are either $A_j$ or $B_j$.
Thus the coordinate minima and maxima are unchanged.  A shortening move
only deletes vertices, proving the second assertion.
\end{proof}

\begin{proposition}[A slab criterion for unbounded births]
\label{prop:slab-criterion}
Fix an integer $w\geq1$ and let
$S_w=\mathbb Z^2\times\{0,1,\ldots,w\}$.
Consider the BFACF graph of polygons
of a fixed knot type $K$ confined to $S_w$, with all moves required to
stay in $S_w$, modulo translations parallel to $S_w$.  If this graph has
infinitely many connected components, then the unrestricted
three-dimensional BFACF filtration for $K$ has strong local-minimum
components at unbounded birth levels.
\end{proposition}

\begin{proof}
Choose a polygon of least length in each confined component.  Any
neutral or shortening BFACF move on a polygon in $S_w$ stays in $S_w$
by \cref{lem:coordinate-box}.  Consequently, a sequence of
neutral moves followed by a shortening move from one of the chosen
polygons would stay in its confined component, contradicting minimality.
It therefore belongs to a strong local-minimum component of the full
graph.  Quotient paths by proper lattice isometries lift to actual
moves from the chosen representative, so this argument applies to
our quotient as well.  For every length bound there are only finitely
many polygons in $S_w$ modulo parallel translations.  Hence the least
lengths of infinitely many confined components are unbounded.
\end{proof}

By \cref{cor:every-knot}, the conclusion of \cref{prop:slab-criterion}
now holds for every knot type, by an unconfined mechanism.  Its hypothesis
remains open and has independent interest as a statement about confined
BFACF graphs.  The confined-class phenomena reported in
\cite{IshiharaSlabs} motivate it, but several classes do not supply
infinitely many.

The computations also suggest the following problems and conjectures.

\begin{problem}[Distribution of BFACF mirror barriers]
For an amphichiral knot type \(K\), determine the set of values
\[
 B_{\SC,\BFACF}(P,r(P))
\]
as \(P\) ranges over minimal simple-cubic representatives and \(r\) ranges over
orientation-reversing cubic symmetries.  The complete trees above imply that,
under the present quotient and BFACF system, every minimal $4_1$ class is
separated from its reflected class by excess barrier $2$, while every minimal
$6_3$ class has reflected-class barrier $4$.  Are such mirror barriers uniformly
bounded in terms of crossing number or minimal lattice length?  Can they become
arbitrarily large for other amphichiral knot types?
\end{problem}

\begin{conjecture}[Symmetry breaking for tight figure-eight knots]
\label{conj:continuum-s4}
Let $\Rop_{S_4}(4_1)$ be the infimum of ropelength over smooth
figure-eight embeddings invariant under an order-four Euclidean
rotoreflection.  Then
\[
 \Rop_{S_4}(4_1)>\Rop(4_1).
\]
\end{conjecture}

The lattice obstruction in \cref{subsec:fourfold-symmetry} motivates
this continuum question but gives no ropelength lower bound.  The
existence of symmetric ropelength-critical configurations, as studied
in \cite{CantarellaEllisFuMastin}, does not imply that they are global
minimizers.  Even a strict inequality in
Conjecture~\ref{conj:continuum-s4} would not by itself prove separation of
mirror-related components of the continuum ideal stratum.

\begin{remark}[Additional data beyond birth]
A merge tree records component incidence as well as birth. To prove that it
is a strictly finer knot invariant requires a pair of knot types with equal
birth level and different merge data; the reported examples alone do not
establish that separation.
\end{remark}

\begin{problem}[Dependence on move system]
Compare the merge ultrapseudometrics obtained from BFACF-type moves and from
cubulated moves.  When do they define the same hierarchical clustering of the
minimal lattice layer?
\end{problem}

\begin{problem}[Dependence on lattice]
Compare the discrete merge scales obtained from the simple cubic, body-centered
cubic, and face-centered cubic lattices.  Which features are lattice-independent?
\end{problem}

\begin{problem}[Stable discrete merge scale]
Does some stabilization procedure, such as lattice refinement or controlled
subdivision, produce a merge invariant of \(K\) that is independent of the chosen
lattice?
\end{problem}

\begin{problem}[Comparison with ropelength]
Find explicit functions \(A,B\) such that lattice merge scales give upper or lower
bounds for ropelength merge scales:
\[
A\,m_{\mathcal L,\mathcal M}(C,D)
\leq
m_{\operatorname{rop}}(C',D')
\leq
B\,m_{\mathcal L,\mathcal M}(C,D),
\]
where \(C',D'\) are suitable continuous counterparts of \(C,D\).  On the
continuous side the target is now unambiguous: by \cite[Theorem~6.20]{OzawaRFKS}
the connected and path merge functions agree, and by
\cite[Theorem~8.2]{OzawaReidGraphs} the common value is computed by the
right-relaxed diagrammatic merge scale in
\(\mathcal G^{\operatorname{lift}}_{\Lambda,u}(K)\).
\end{problem}

\begin{problem}[Finite recognition length from lattice graphs]
Use lattice-filtered move graphs to define a finite recognition length
\(L_{\operatorname{char}}^{\mathcal L}(K)\): the first lattice length level at
which a finite diagrammatic pattern recognizing \(K\) appears in the projection
of the lattice graph, and compare it with the continuous recognition length
\(L_{\operatorname{char},u}(K)\) of \cite[Theorem~7.14]{OzawaReidGraphs}.
\end{problem}

\section{Conclusion}
\label{sec:conclusion}

We have organized lattice-filtered move graphs as a finite-state experimental
model for quantitative reconfiguration of knot representatives.  The underlying
move-graph idea and the BFACF dynamics are classical; the contribution here is to
place them under an explicit length filtration and to compute the first level at
which prescribed components merge.  The resulting merge hierarchy is the usual
ultrametric structure associated with nested component partitions, while the
knot-theoretic content lies in the minimax heights and their finite certificates.

The main experimental conclusion is that standard BFACF dynamics exhibit exact
bounded-height reconfiguration barriers not visible from ordinary knot type alone.
The classical ergodicity theorem guarantees eventual merging for every fixed
pair of configurations, but not eventual connectivity of the whole capped
graph.  Our computations on the reported shortest-layer collections determine how much temporary
length increase is
needed to merge the birth-level components for two amphichiral knots.  For
$4_1$, four components of sizes $58,58,18,18$ at $N=30$ merge simultaneously at
$N=32$.  For $6_3$, twelve components at $N=40$ occur as six explicit mirror pairs
$A_r\leftrightarrow B_r$; the six $A_r$ components merge to a branch $A$ at
$N=42$, the six $B_r$ components merge to its mirror branch $B=\mu(A)$, and
$A$ and $B$ merge at $N=44$.  Thus $4_1$ has a one-stage hierarchy with
barriers $0,2$, whereas $6_3$ has a genuine two-stage, mirror-symmetric hierarchy
with barriers $0,2,4$.  The archived seed-to-mirror paths
provide finite certificates for the maximal barriers, while exhaustive bounded
searches establish the global component partitions.

The trefoil computation goes beyond the minimal layer.  Six verified
strong local-minimum components are born at $26$; five merge with the
minimal branch at $28$, while the $20$-state component $D_4$ first
merges at $30$.  The nine selected births therefore give the tree in
\cref{fig:trefoil-selected-tree} and the persistence submodule in
\cref{cor:trefoil-selected-tree}.  Their individual certifications are
unconditional; exhaustiveness through length $26$ requires the separate
population hypothesis in \cref{prop:trefoil26-completeness}.  In addition, an explicit $64$-edge Hamiltonian unknot is a frozen
nonminimal component whose exact merging height with the square is $66$.
A two-edge splice gives Hamiltonian unknots $P_m$ of length $64m$ with
no neutral or shortening BFACF moves for every $m\ge1$.  A protected-port
connected sum with a minimal $24$-edge trefoil similarly produces an
$88$-edge frozen trefoil and frozen trefoils of lengths $64m+24$ for every
$m\ge1$; reflection gives both chiralities.  Only four of the $2448$
port-aligned joins of minimal trefoils to the block are frozen, so the
protected port is a genuine geometric constraint.  Thus strong birth levels
are unbounded for both the unknot and both trefoil chiralities, and their
capped BFACF graphs fail to be eventually connected.  The general finite-sublevel criterion
in \cref{prop:eventual-connectivity} now has a decisive example on the
nonconnected side.  The block substitution of \cref{subsec:frozen-all}
removes the restriction to these two knot types.  Every vertex of a
lattice polygon $Q$ is replaced by a $5\times5\times5$ Hamiltonian
block with face-centred ports and normal end edges.  The result is a frozen
polygon of the same oriented knot type and length $125\,\ell(Q)$.
Consequently, for every knot type, strong birth levels are unbounded,
lattice Gordian configurations exist, and the capped BFACF graphs fail to
be eventually connected.  Unbounded \emph{excess}, rather than mere
unbounded birth level, remains open.

The low-crossing comparison also shows why rigidity and mirror
separation should be recorded separately: the supplied $48$-edge $8_3$
polygon is an isolated capped vertex but is fixed by reflection in the
proper-cubic quotient.  Its zero mirror barrier contrasts with the
positive barriers of the two amphichiral minimal layers computed above.

The next step is to extend certified merge computations to further knot
types and higher birth levels, and to determine which geometric features
of a component constrain its merge height.

\section*{Code and data availability}

The original computational supplement provides a prototype Python
implementation.  The present source package retains the earlier summary
tables and figures and adds the self-contained trefoil certificate bundle
described below.  The script constructs
BFACF-filtered move graphs from supplied finite lists of simple cubic lattice
polygons, computes connected components, and records summary data.  The
original supplementary archive contains the main script \texttt{lattice\_bfacf\_graph.py},
the merge-check script \texttt{merge\_check\_bfacf.py}, the braid-seed helper
\texttt{make\_braid\_seed.py}, a README file with the reproduction commands, and
the JSON seed and summary files used in the experiments.

The original supplementary archive includes the trefoil seed; the figure-eight seed, its reflected
mirror seed, and the combined mirror-comparison input; the minimal $6_3$ seed,
its reflected mirror, the source-coordinate file and provenance record; the
\(5_1\) seed and the non-minimal braid-generated \(5_2\) stress-test seed; and
the granny/square knot sanity-check seeds.  The JSON summaries record the
computations reported in the text, including the mirror merge checks for $4_1$
at \(N=32\) and $6_3$ at \(N=40,42,44\).  The archive includes both explicit
merge certificates, their summary and verification files, and the certificate
figures used in Figures~\ref{fig:figure8-merge-certificate} and
\ref{fig:six3-merge-certificate}.  It also contains the deterministic C++17
accelerator \texttt{bfacf\_fast.cpp}, the generic verifier
\texttt{verify\_merge\_certificate.py}, and the figure-generation script.
The second-seed material comprises the seed-construction script
\texttt{make\_six3\_seed.py} (whose recorded RNG seed regenerates the archived
\texttt{six3b} seed), the knot-type verifier \texttt{verify\_knot\_type.py}
with its JSON verification outputs, the pure-Python optimized search script
\texttt{merge\_check\_bfacf\_fast.py}, the \texttt{six3b} seed and summary
files, the verified \(157\)-state \(N=44\) certificate
\texttt{six3b\_merge\_path\_44.json} with its own certificate figure, and the
verified \(15\)-move level-42 cross path \texttt{six3\_cross\_path\_42.json}
connecting the two seed pairs.

The implementation is intended as a transparent prototype rather than an
optimized enumeration engine.  For the earlier amphichiral-knot experiments,
minimal representative lists are supplied from the published lattice-knot data.
The source package additionally contains
\path{figure8_minimal_components.csv},
\path{figure8_merge_tree_summary.txt}, and
\path{six3_minimal_component_summary.csv}, together with the two complete
merge-tree figures.  These record the global birth-level partitions and the
subsequent merge levels used in the $4_1$ and $6_3$ theorems.  The earlier JSON summary
files retain the seed-generated and capped searches used for the explicit path
certificates.  The merge-check script returns one of three outcomes at a tested
level: a connecting state is found, one frontier is exhausted and hence the two
seeds are separated at that level, or a state/time cap is reached and the test is
inconclusive.

\paragraph{Hamiltonian frozen-unknot certificates.}
The directory \path{anc/frozen_unknots/} contains the coordinate-based
generator \path{construct_family.py}, JSON vertices for representative
values $m=1,2,3,4,10$, the original $64$-edge BFACF path and the
independent checker \path{base64/verify_bfacf_certificate.py}.
To reproduce the two distinct claims, run:
\begin{verbatim}
cd anc/frozen_unknots
python3 construct_family.py --m 2
python3 construct_family.py --m 10
python3 base64/verify_bfacf_certificate.py \\
    base64/gordian64_bfacf_64_to_4_cap66.txt
\end{verbatim}
The first two commands check closure, complete vertex occupancy,
seam bridges, and absence of three-edge plaquettes; the last checks
all $25{,}436$ transitions, the start and end lengths, the maximum
length $66$, and the absence of neutral and shortening moves from
the base block.  The full infinite-family conclusion also uses the
symbolic local splice proof in
\cref{thm:infinitely-many-frozen-unknots}.
The $m=1$ path certificate does not establish the exact merge height
for $m>1$.  The figures and the LaTeX stand-alone construction note
are archived with the ancillary data.

For the bidirectional merge script, the fields
\texttt{states\_from\_seed\_1} and \texttt{states\_from\_seed\_2} count discovered
states in the two search frontiers at the stopping time.  The field
\texttt{expanded} counts states actually expanded from the queues.  Thus these
quantities measure different aspects of the search and should not be added.  The
runtime field is hardware-dependent; the archived README records the Python and
machine information used for the supplied summaries.  The original supplementary code and data are archived on Zenodo at DOI:
\texttt{10.5281/zenodo.21304711}.  The concept DOI
\texttt{10.5281/zenodo.20370840} always resolves to the latest version of
that record.  The additions in this revision are supplied in the present
source package.  The November 2026 Zenodo version adds the directory
\path{frozen_all_knots/}; earlier versions do not contain it.

\paragraph{Protected-port frozen-trefoil certificates.}
The directory \path{anc/frozen_trefoils/} contains the minimal $24$-edge
cyclic word, the $88$-edge vertex list, the stand-alone verifier
\path{verify_connected_sum_public.py}, and JSON test summaries.
The verifier reconstructs the joined polygon from the original two
words for $m=1,2,3,4,8,16$ and checks the entire local move set by
coordinate enumeration:
\begin{verbatim}
python3 anc/frozen_trefoils/verify_connected_sum_public.py
\end{verbatim}
The two-point separating-plane argument determines its knot type;
the local port-check proof in
\cref{thm:infinitely-many-frozen-trefoils} gives the universal conclusion.
The supplied finite samples alone do not establish the statement for
all $m$.  The files are also included in the companion Zenodo update.

\paragraph{Frozen polygons of every knot type.}
The directory \path{anc/frozen_all_knots/} contains the two base block
paths of \cref{lem:base-blocks} as coordinate lists and direction words
(\path{blocks_S5.json}).  The independent move generator and projection
code is \path{lattice_tools.py}, and the construction of
\cref{def:block-substitution} is implemented in \path{construct.py}.  The
script \path{arc_certificates.py} produces the unknot certificates for the
closed base arcs, and \path{verify_family.py} produces the checks of
\cref{tab:frozen-all-checks}.  The SAT search that found the block paths is
included for provenance (\path{sat_unknotted.py}).  It rejected the first
few Hamiltonian solutions because their arcs were knotted.  Condition~(B4)
is therefore a genuine constraint, not a formality.  To reproduce:
\begin{verbatim}
cd anc/frozen_all_knots
python3 construct.py          # all 30 oriented block paths
python3 arc_certificates.py   # (B4) for the two base paths
python3 verify_family.py      # Table of block substitutions
\end{verbatim}
The first two scripts check every hypothesis used in the proof of
\cref{thm:frozen-every-knot}.  The third is an illustration only.  The
knot-type checks use SnapPy~\cite{SnapPy} and require it to be installed.

\paragraph{Trefoil certificates added in this revision.}
The directory \path{anc/trefoil26/} contains the complete vertex and
parent lists for $A,B,E$, the six closed lists $D_i$, and six paths
from a common minimal root.  The path files are named
\path{certificate_C1.txt} through \path{certificate_C6.txt}; their suffixes
refer to $D_1,\ldots,D_6$, not to the minimal components $C_j$.
Direction digits $0,1,2,3,4,5$ mean $+x,+y,+z,-x,-y,-z$.
The Python coordinate verifier checks the six neutral components and
all path transitions.  The separate C++ coordinate audit checks every
state and every legal capped neighbor in each of $A,B,E$, using a
canonicalization routine based on explicit proper rotations.  It also
checks neighbor-set equality with the search implementation, parent
edges and connectivity.  The C++ implementations share direction
conventions and low-level data types; the independent Python checker
separately implements coordinates and canonicalization.
From that directory, the essential verification commands are:
\begin{verbatim}
python3 verify_certificates.py
g++ -O3 -std=c++17 audit_coordinate.cpp -o audit_coordinate
./audit_coordinate old26_parent.tsv 26
./audit_coordinate reached28_parent.tsv 28
./audit_coordinate outside_cap28_0.tsv 28
python3 check_disjoint.py
python3 verify_manuscript_data.py
python3 verify_geometry.py
python3 verify_figure_path.py
\end{verbatim}
Assertions must remain enabled: do not use Python's \texttt{-O} or
the C++ option \texttt{-DNDEBUG}.  The bundle includes recorded audit
outputs, summary tables and SHA-256 hashes.  Exploratory walks and
truncated searches are retained as discovery provenance only; none
is used as a separation or completeness certificate.  In particular,
the proof at cap $30$ uses an explicit path, not an exhaustive search
of $G_{30}$.

\paragraph{Figures and low-crossing checks.}
The script \path{anc/trefoil26/draw_trefoil_path.py} lifts and draws the
reversed trefoil certificate.  Its coordinate output is checked by
\path{verify_figure_path.py}, which tests unit-plaquette edge differences.
The supplement includes \path{trefoil_all_47_states.pdf} and the fixed
projection matrix.  The directory \path{anc/low_crossing/} contains
the unchanged downloaded $8_3$ and $8_7$ coordinate and symmetry files,
their source URLs, \path{verify_seeds.py}, and its JSON summary.
The latter computation certifies the two specified components and
the $8_3$ affine symmetry; it is not an independent knot-recognition
algorithm or an exhaustive enumeration of either full knot space.
Run \texttt{python3 verify\_seeds.py} from that directory to reproduce
these checks.  The file \path{anc/README_ARXIV.md}
maps the new claims to their certificate files and verification
commands, and SHA-256 manifests for checking file integrity.  The trefoil
span check in \cref{rem:trefoil-spans} uses the same archived component
lists as the closure certificates.
The directory \path{anc/gordian_review/} records the literature sources
and the model distinctions used in \cref{subsec:gordian-literature},
with a reusable BibTeX bibliography.  The new structural propositions
in \cref{subsec:eventual-connectivity} have self-contained proofs and
do not depend on an additional numerical search.

\paragraph{Independent re-implementation.}
The directory \path{anc/independent_recheck/} contains a BFACF move
generator and canonicalization written from scratch, without reusing the
search or audit code above, together with scripts that recompute the
principal claims from the archived data: the four type-$0$ components
$58,58,18,18$ of the $152$ published minimal $4_1$ classes, their mirror
pairing, and the $19{,}618$-state cap-$32$ component; the trefoil
component $A$, the six closed components $D_i$ with their sizes and edge
counts, the six path certificates and the cap-$28$ components $E$ and $B$;
all $25{,}436$ transitions of the $64$-edge unknot certificate; and the
port search of \cref{rem:port-search}.  The directory also contains the
drawing script for \cref{fig:frozen-trefoil88}, recorded outputs and a
SHA-256 manifest.  The $6_3$ computations are not rechecked there, since
the $148$ minimal coordinate lists belong to the original Zenodo archive.

\section*{Declaration of generative AI and AI-assisted technologies}

During the preparation of this manuscript, the author used ChatGPT
(OpenAI, GPT-5.5 Thinking and GPT-5.6 Thinking) and Claude (Anthropic) for
language polishing, organizational suggestions,
preliminary consistency checks, assistance in drafting some explanatory text
and code comments, and assistance in preparing and running the \(6_3\)
computations and the associated verification scripts.  ChatGPT (OpenAI,
Codex) was also used to develop and run the additional trefoil searches,
coordinate-based audits and path verifiers, and to integrate their results
and the local-minimum distinctions into this revision.  ChatGPT (OpenAI, GPT-6) was additionally used to inspect and re-run
the frozen-unknot certificates, independently check the Hamiltonian
splice family, and assist in integrating the verified certificate,
construction, and splice argument into this revision.  AI-assisted
coordinate checking and preparation of the protected-port trefoil
construction, its reproducibility scripts, and accompanying exposition
were also performed for the October 2026 revision.  Claude (Anthropic)
was used in the same revision to write the independent re-implementation
in \path{anc/independent_recheck/}, to rerun the protected-port search
reported in \cref{rem:port-search}, to draw
\cref{fig:frozen-trefoil88}, and to check the consistency of the
ancillary files with the text.  In the November 2026 revision, Claude
(Anthropic, Claude Opus~5.5) proposed the block-substitution construction of
\cref{subsec:frozen-all} and the parity explanation for side length five.
It found the two base block paths by a SAT search, wrote the verification
scripts in \path{anc/frozen_all_knots/}, drafted the statements and proofs
of \cref{lem:base-blocks,lem:trivial-arcs,thm:frozen-every-knot} and
\cref{cor:every-knot}, and revised \cref{sec:further} accordingly.  All AI-assisted
material is subject to the author's review and verification.  The author
takes full responsibility for the final content of this manuscript.

\end{document}